\documentclass[11pt,letterpaper]{amsart}

\usepackage{amsfonts,amsthm,latexsym,mathtools,amsmath,amssymb,enumitem,appendix,color,hyperref}
\usepackage[capitalize]{cleveref}
\usepackage{float}
\usepackage{pst-node}
\usepackage{tikz-cd} 
\usepackage{appendix}

\usepackage{graphicx}
\DeclareGraphicsExtensions{.pdf,.png}
\graphicspath{{./figures/}{./}}

\usepackage{comment}

\makeatletter
\@namedef{subjclassname@2020}{\textup{2020} Mathematics Subject Classification}
\makeatother

\theoremstyle{plain}
\newtheorem{theo+}{Theorem}
\numberwithin{theo+}{section}
\newtheorem{prop+}[theo+]{Proposition}
\newtheorem{coro+}[theo+]{Corollary}
\newtheorem{lemm+}[theo+]{Lemma}

\theoremstyle{definition}
\newtheorem{defi+}[theo+]{Definition}
\newtheorem{problem}[theo+]{Problem}

\newtheorem*{pb-prime}{Problem 1$\mathbb{'}$}

\newtheorem{not+}[theo+]{Notation}

\theoremstyle{remark}
\newtheorem{rema+}[theo+]{Remark}

\newenvironment{theorem}{\begin{theo+}}{\end{theo+}}
\newenvironment{proposition}{\begin{prop+}}{\end{prop+}}
\newenvironment{corollary}{\begin{coro+}}{\end{coro+}}
\newenvironment{lemma}{\begin{lemm+}}{\end{lemm+}}
\newenvironment{remark}{\begin{rema+}}{\end{rema+}}
\newenvironment{definition}{\begin{defi+}}{\end{defi+}}
\newenvironment{notation}{\begin{not+}}{\end{not+}}

\newcommand{\defin}[1]{%
\relax\ifmmode%
\textcolor{blue}{#1}%
\else \textcolor{blue}{\emph{#1}}%
\fi%
}

\newcommand{\setC}{\mathbb{C}}

\newcommand{\bC}{\setC}

\renewcommand{\Re}{\operatorname{Re}}
\renewcommand{\Im}{\operatorname{Im}}

\newcommand{\B}{\mathcal{B}}
\newcommand{\PP}{\mathcal{P}}
\newcommand{\ZZ}{\mathcal{Z}}
\newcommand{\F}{\mathcal{F}}

\newcommand*\diff{\mathop{}\!\mathrm{d}}

\numberwithin{equation}{section}

\newcommand {\Ga} {\Gamma}

\newcommand{\bZ}{\mathbb Z}

\newcommand{\N}{\mathcal N}

\newcounter{margnotes}

\begin{document}

\title[The translation geometry of P\'olya's shires]
{The translation geometry of P\'olya's shires}

\author[R.~Bøgvad]{Rikard Bøgvad}
\address[Rikard Bøgvad]{Department of Mathematics, Stockholm University, SE-106 91
Stockholm, Sweden}
\email{rikard@math.su.se}

\author[B.~Shapiro]{Boris Shapiro}
\address[Boris Shapiro]{Department of Mathematics, Stockholm University, SE-106 91
Stockholm, Sweden \& Department of Mathematics, Guangdong Technion - Israel Institute of Technology, 241 Daxue Road, Shantou, Guangdong, China, 515063}
\email{shapiro@math.su.se}

\author[G.~Tahar]{Guillaume Tahar}
\address[Guillaume Tahar]{Department of Mathematics with Computer Science, Guangdong Technion - Israel Institute of Technology (GTIIT), Shantou, Guangdong, China}
\email{guillaume.tahar@gtiit.edu.cn}

\author[S.~Warakkagun]{Sangsan Warakkagun}
\address[Sangsan Warakkagun]{Department of Mathematics, Khon Kaen University, Khon Kaen, Thailand}
\email{sangwa@kku.ac.th}

\date{\today}

\begin{abstract} 
In his shire theorem, G. P\'olya proves that the zeros of iterated derivatives of a meromorphic function in the complex plane accumulate on the union of edges of the Voronoi diagram of the poles of this function. By recasting the local arguments of P\'olya into the language of translation surfaces, we prove its generalisation describing the asymptotic distribution of the zeros of a meromorphic function on a compact Riemann surface under the iterations of a linear differential operator $T_\omega: f \mapsto \frac{df}{\omega}$ where $\omega$ is a given meromorphic $1$-form. The accumulation set of these zeros is the union of edges of a generalised Voronoi diagram defined by the initial function $f$ together with the singular flat metric on the Riemann surface induced by $\omega$. This result provides the ground for a novel approach to the problem of finding a flat geometric presentation of a translation surface initially defined in terms of algebraic or complex-analytic data.
\end{abstract}

\maketitle

\tableofcontents

\section{Introduction}\label{basicq}

\subsection{Short historical account}

The classical shire theorem of G.~P\'olya claims that for a meromorphic function $f$ with the set $S$ of its poles, the  zeros of its iterated derivatives $f^{(n)}$ accumulate when $n\to +\infty$ along the edges of the Voronoi diagram associated with $S$, see~\cite{Hay}.  
An illustration of this famous result is shown in Figure~\ref{vv3}.
\begin{figure}[!hbt]
\centering
\includegraphics[width=0.55\linewidth]{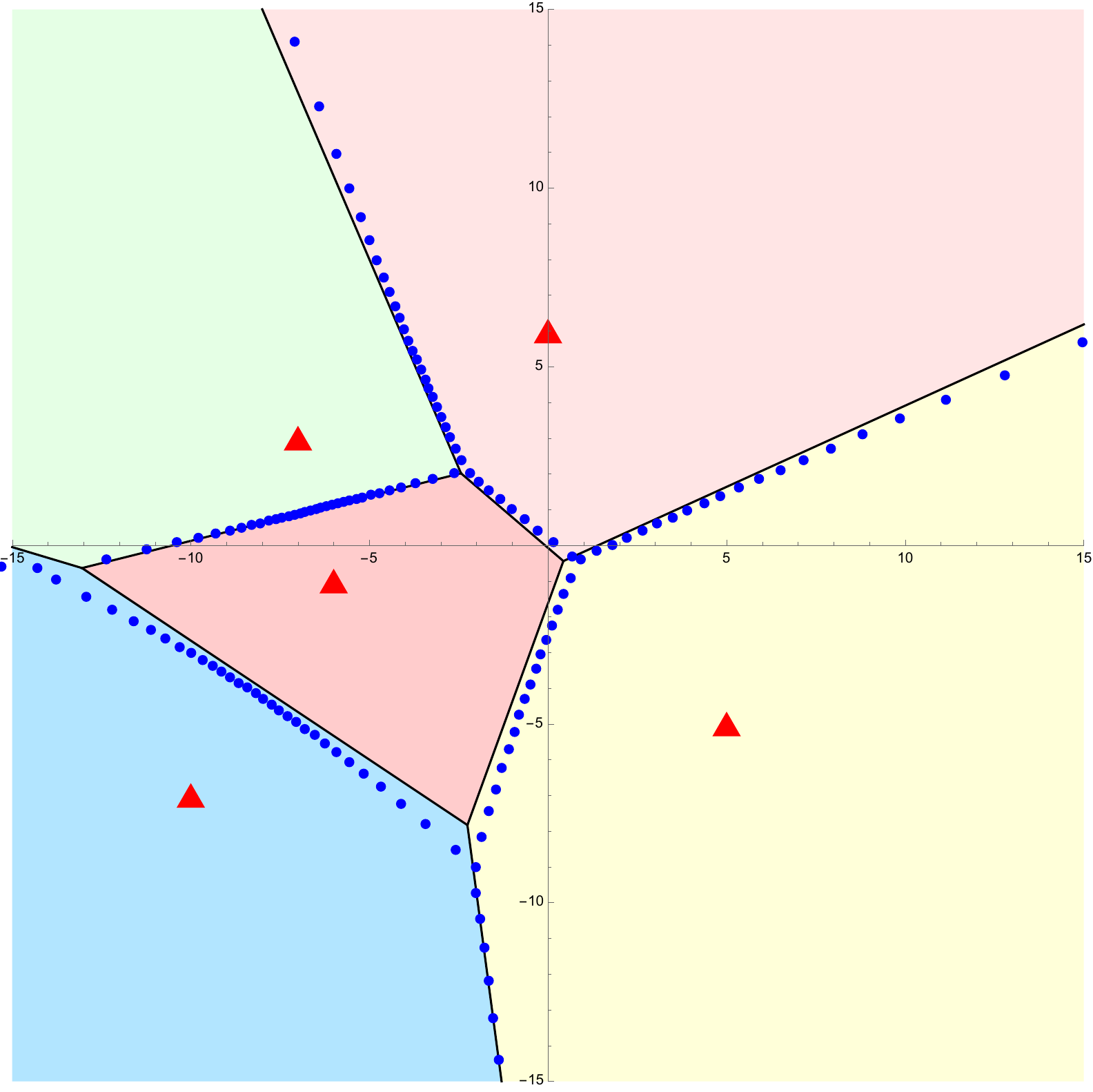}
\caption {The Voronoi diagram determined by the five poles (red triangles) 
of some rational function $f$, together with the zeros of $f^{(20)}$ in blue dots.}\label{vv3}
\end{figure}
 
Several prominent mathematicians including N.~Wiener, E.~Hille, and R.~P.~Boas have continued P\'olya's line of study soon after the publication of the latter theorem, see references in ~\cite{Pol1}. Over the years a number of articles extending and generalising the original result has appeared, see e.g. \cite{ClEd, Ge, Ge2, PrSh1, PrSh2, Rob}.

More recent publications have concentrated on the weak limits of the root-counting measures for the zeros of $f^{(n)}$. In particular, Ch.~H\"agg and R. B\o gvad obtained a measure-theoretic refinement of P\'olya's shire theorem for rational functions, see~\cite{BoHa}. Using currents they also proved a similar result for Voronoi diagrams associated with generic hyperplane arrangements in $\mathbb{C}^{m}$.

Later Ch.~H\"agg extended the main result of~\cite{BoHa} by considering meromorphic functions of the form $f=Re^U$ where $R$ is a rational function with at least $2$ distinct poles and $U$ is a non-constant polynomial, see~\cite{Ha}. The class of such functions coincides with the class of meromorphic functions which are quotients of two entire functions of finite order, each having a finite number of zeros, see~\cite{Ti}.

\medskip
In  ~\cite{Ke}, V.~Keo extended the results of ~\cite{BoHa} and ~\cite{Ha} by studying a particular meromorphic function of a different type, namely, $f(z)=1/\left(1-e^z\right)$, which has an infinite number of poles and whose iterated derivatives are related to the Eulerian polynomials. In addition, he considered iterations of rational functions under the action of differential operators of the form $\mathcal{D}=g(z)\dfrac{\partial}{\partial z}$, where $g(z)$ is a polynomial in $z$, and studied in details the differential operator $z\dfrac{\partial}{\partial z}$, as well as formulated some conjectures. These cases are closely related to the topic of the present paper.

Finally, in \cite{Wei}, M.~Weiss provided  a generalisation of P\'olya's classical theorem to automorphic functions in the half-plane.

\subsection{Our set-up}

There are (at least) two natural ways to generalise the geometry of the complex plane to surfaces of higher genus such as:
\begin{itemize}
    \item hyperbolic surfaces (the metric still has constant curvature,  but it is no longer flat);
    \item translation surfaces (the metric is still flat, but it has conical singularities).
\end{itemize}

Below we generalise P\'olya's shire theorem to the case of meromorphic functions on compact Riemann surfaces equipped with a flat metric with conical singularities and study a class of linear differential operators corresponding to the complex-analytic data defining a translation structure, i.e. we use the second of the above generalisations. For the background on translation surfaces see \cite{Zor}.

Namely, let $X$ be a compact Riemann surface with a fixed meromorphic $1$-form $\omega$. We associate to the pair $(X,\omega)$ the linear differential operator $T_{\omega}$ acting on meromorphic functions on $X$ as 
\begin{equation}\label{eq:oper}
T_{\omega}: f\mapsto \frac{df}{\omega}.
\end{equation}

Now given a meromorphic function $f$ on $X$, we are interested in the asymptotic of zeros for the sequence $(f_{n})_{n\in \mathbb{N}}$ of meromorphic functions  defined inductively as 
$$f_{0}=f,~f_{n+1}=T_{\omega}(f_{n})=T_{\omega}^{n+1}(f),~n\geq 0. $$

\medskip
\noindent
{\bf Terminology.}
In this article, we consider meromorphic functions and differentials, as well as Puiseux series. For this reason, we will adopt a mixed terminology:
\begin{itemize}
    \item When a point is \textit{explicitly} designated as a zero, its order is taken to be a positive integer.
    \item When a point is \textit{explicitly} designated as a pole, its order \textit{as a pole} is taken to be a positive integer. In particular, a simple pole is a pole of order $1$.
    \item When a point is designated more generally as a singularity, its order may be either positive or negative. If the number is negative, then the singularity is in fact a pole.
\end{itemize}

\medskip
\begin{definition}\label{defn:limitset}
For a meromorphic function $f$ on a Riemann surface $X$ and any operator $T$ acting on the space of meromorphic functions, define the \textcolor{blue}{\textit{limit set}} $\mathcal{L}(T,f)\subseteq X$  as the set of points $z \in X$ such that any open neighbourhood of $z$ in $X$ contains a zero of $T^{n}(f)$ for infinitely many $n$.
\end{definition}

In the classical shire theorem, the limit set coincides with the Voronoi diagram in $\mathbb{C}$ associated with the set of poles of a meromorphic function. In our generalised settings, the Voronoi diagram of a meromorphic function on a Riemann surface is defined with respect to the so-called \textcolor{blue}{\textit{principal polar locus}} and the singular flat metric induced by the translation structure on the surface.

\begin{definition}\label{defn:localfactor}
Consider a (compact) Riemann surface $X$ with a meromorphic $1$-form $\omega$. 
Given a point $z_{0}\in X$ which is not a pole of $\omega$, we say that a meromorphic function $f$ is \textcolor{blue}{\textit{locally factorised}} by a primitive of  $\omega$ having no pole at $z_0$ if there exist:
\begin{itemize}
    \item a neighbourhood $U$ of $z_{0}$ in $X$,
    \item a holomorphic function $\phi$ defined on $U$;
    \item a holomorphic function $g$ defined on a neighbourhood of $\phi(z_{0})$ in $\mathbb{C}$
\end{itemize}
such that $\omega = d\phi$ and $f = g \circ \phi$ in $U$.
\end{definition}

Note that the property above is independent of any choice of branch at a zero of $\omega$, since any local primitive of $\omega$ extends as a single-valued function on a neighborhood of the zero.

Below by a \textcolor{blue}{\textit{non-zero meromorphic $1$-form}} we always mean a $1$-form not vanishing identically on the underlying Riemann surface. 

\begin{definition}\label{defn:PPL}
Consider a non-zero meromorphic $1$-form $\omega$ and a fixed meromorphic function $f$ on a compact Riemann surface $X$. The \textcolor{blue}{\textit{principal polar locus}} $\mathcal{PPL}(\omega,f)$ of the pair $(\omega,f)$ is the subset of $X$ containing:
\begin{itemize}
    \item the poles of $f$ that are not poles of $\omega$;
    \item the zeros of $\omega$ at which  $f$ is not locally factorised by a primitive of $\omega$ (see Definition~\ref{defn:localfactor}).
\end{itemize}
\end{definition}

\begin{remark}
In the original shire theorem, the Riemann surface is the extended complex plane $\mathbb{CP}^1=\mathbb{C} \cup \lbrace{ \infty \rbrace}$, $\omega=dz$ and $\mathcal{PPL}(\omega,f)$ is the set of the affine poles of $f$ (i.e. poles different from $\infty$).

\end{remark}

The main result of our paper is as follows. 

\begin{theorem}\label{thm:MAIN}
Consider a non-zero meromorphic $1$-form $\omega$ on a compact Riemann surface $X$, its associated differential operator $T_{\omega}$
 and any meromorphic function $f$ on $X$ such that $\mathcal{PPL}(\omega,f) \neq \emptyset$. Then the following two facts are valid. 

\medskip
\noindent
\rm{(i)} The limit set $\mathcal{L}(T_{\omega},f)$ is the union of:
\begin{itemize}
    \item the Voronoi diagram $\mathcal{V}_{\omega,f}$ defined by $\mathcal{PPL}(\omega,f)$ (see Definition~\ref{defn:Voronoi} below);
    \item the poles of $\omega$ of order at least two;
    \item the simple poles of $\omega$ that are not poles of $f$. 
\end{itemize} 

\medskip
\noindent
\rm{(ii)} The asymptotic root-counting measure of the sequence $(T_{\omega}^{n} (f))_{n \in \mathbb{N}}$ is given by 
$$\frac{\mu_{\omega,f}}{A}+\frac{1}{A} \sum\limits_{p \in \mathcal{P}} (d_{p}-1)\delta_{p},$$
where $\mu_{\omega,f}$ is the Cauchy measure of the Voronoi diagram $\mathcal{V}_{\omega,f}$ (see Definition~\ref{defn:Cauchy} below), $\mathcal{P}$ is the set of poles of $\omega$, and $A = \sum\limits_{z \in \mathcal{PPL}(\omega,f)} (a_{z}+1)$ where $z$ is a zero of $\omega$ of order $a_{z}$ and $p \in \mathcal{P}$ is a pole of $\omega$ of order $d_{p}$.
\end{theorem}

\begin{remark}
Pólya's original setup covers the case of meromorphic functions in $\mathbb{C}$ with at least two poles and possibly an essential singularity at $\infty$. In contrast, applying Theorem~\ref{thm:MAIN} to $\mathbb{CP}^{1}$ with $\omega=dz$, we cover only the case of rational functions. However, following our proof, one might extend Theorem~\ref{thm:MAIN} to the case of meromorphic functions with possibly essential singularities at any of the poles of $\omega$. We describe this extension briefly in \S~7.1.
\par
Contrary to Pólya's setting, our meromorphic functions do not need to have two poles. We only retain the assumption that the $\mathcal{PPL}(\omega,f)$ is nonempty. It may happen that the Voronoi diagram is trivial, in which case the zeros accumulate only at the poles of $\omega$.
\end{remark}

For a translation surface defined in terms of complex analysis (via a Fuchsian group and a modular form) or algebraic geometry, finding a presentation in terms of flat geometry—specifically, as a polygon with identified edges—is generally a difficult problem. A crucial outcome of Theorem~\ref{thm:MAIN} is that it determines the incidence structure of the cells of the Voronoi diagram from analytic data. In principle, computing the periods of the relative homology classes of arcs dual to the edges of this diagram is sufficient to construct a flat model. We discuss this perspective in more detail in \S~7.2.

\begin{remark}
The Cauchy measure that we define on the edges of the Voronoi diagram is not  invariant under the directional flow of the translation surface. However, it is continuous with respect to the standard Lebesgue measure induced by the flat metric.
\end{remark}

To conclude this introduction, we present three explicit examples of growing complexity, designed to showcase the scope and applicability of the theory.

\subsection{A first concrete example}

\medskip
\noindent
The following example is simple enough to allow an exact description of the zero distribution, rather than only an asymptotic one, explicitly showing how the Cauchy measure arises in the asymptotic root-counting measure.

\medskip
\noindent We consider iterated derivatives of $f(z)=\frac{2i}{z^{2}+1}$. Here, $f$ is a meromorphic function on $\mathbb{CP}^{1}$ and $\omega=dz$. The principal polar locus $\mathcal{PPL}(\omega,f)$ is formed by the two poles $\pm i$ of $f$. The Voronoi diagram $\mathcal{V}_{\omega,f}$ is the perpendicular bisector of $[-i,i]$ for the standard Euclidean metric induced by $dz$ and thus coincides with the real axis $\mathbb{R}$. As $i$ and $-i$ are regular points of $\omega$, the normalization factor $A$ (defined in Theorem~\ref{thm:MAIN}) is equal to $2$.

\medskip
\noindent {\bf Explicit determination of the zeros.}

For any $n \in \mathbb{N}$, we have
$f^{(n)}(z)=\frac{(-1)^{n}n!}{(z-i)^{n+1}}-\frac{(-1)^{n}n!}{(z+i)^{n+1}}.$
As a meromorphic function on $\mathbb{CP}^{1}$, $f^{(n)}$ has:
\begin{itemize}
    \item two poles of order $n+1$ at $i$ and $-i$;
    \item one zero of order $n+2$ at $\infty$;
    \item $n$ zeros (counted with multiplicity) in $\mathbb{C}$.
\end{itemize}
These $n$ zeros satisfy the equation $|z-i|^{n+1}=|z+i|^{n+1}$, so they lie in the perpendicular bisector of $[-i,i]$ i.e. the real axis $\mathbb{R}$, see Figure \ref{fig:early-ex}.
\par
For $z \in \mathbb{R}$, writing $z-i=-i\rho e^{i \theta}$ for $\rho>0$ and $-\frac{\pi}{2}<\theta<\frac{\pi}{2}$, we obtain $z+i=i\rho e^{-i\theta}$. The equation $f^{(n)}(z)=0$ then reduces to $e^{2i(n+1)\theta}=(-1)^{n+1}$.
\par
It follows that $f^{(n)}$ has $n$ simple zeros in $\mathbb{R}$ given by formula $\tan(\theta)$ where $-\frac{\pi}{2}<\theta<\frac{\pi}{2}$ and the following congruence is satisfied:
\begin{itemize}
    \item $\theta \equiv 0~[\pi/(n+1)]$ if $n$ is odd;
    \item $\theta \equiv \frac{\pi}{2n+2}~[\pi/(n+1)]$ if $n$ is even.
\end{itemize}
Real zeros are equally spaced with respect to the Cauchy measure that is the projection of the uniform angular measure at $i$ (or equivalently $-i$).

\medskip
\noindent {\bf Asymptotic behaviour.}

As $n$ tends to infinity, the root-counting measure weakly converges to a distribution such that:
\begin{itemize}
    \item $\lim\limits_{n \to +\infty} \frac{n+2}{2n+2}=\frac{1}{2}$ represents the proportion of the total multiplicity of zeros contributed by the zero at infinity;
    \item the measure of a segment of the real axis is the angular measure of the corresponding sector at $i$ (or $-i$) normalized by $2\pi$;
    \item the total measure of the real axis is $\frac{1}{2}$ because it is the projection of the uniform angular probability measure restricted to a half-circle.
\end{itemize}
This is in line with the predictions of Theorem~\ref{thm:MAIN} prescribing the asymptotic root-counting measure to be
$
\frac{\mu_{\omega,f}}{2}+\frac{\delta_{\infty}}{2}
$
.
\par
Compared with the classical theorem of Pólya, our approach provides the degrees of the zeros at infinity (at the poles of the differential $\omega$), which constitutes additional information.

\begin{figure}[h]
\includegraphics[width=0.7\linewidth]{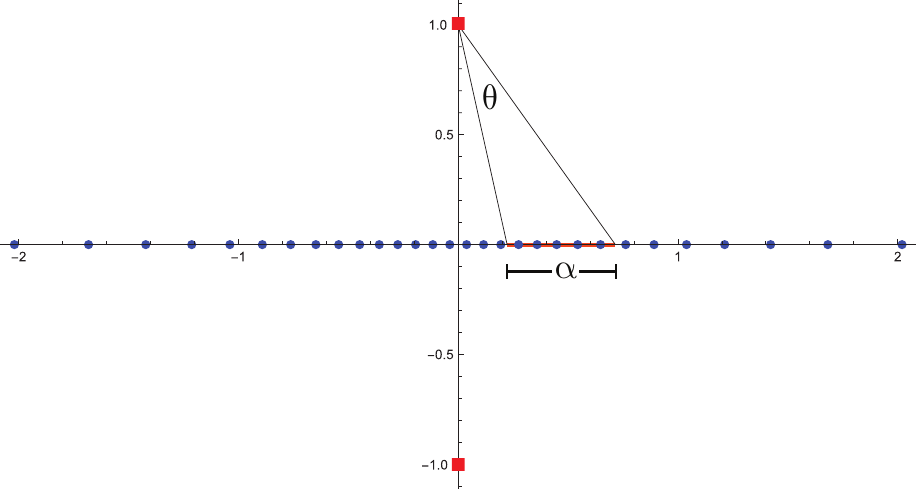}
    \caption{Zeros of $f^{(40)}(z)$ lie on the real line $\mathbb{R}$ in blue dots. The asymptotic Cauchy measure of the segment $\alpha$ is the proportion out of $2\pi$ of the angle $\theta$ from either pole $\pm i$ (in red squares) subtended by the segment.}
    \label{fig:early-ex}
\end{figure}

\subsection{A second example: a monomial linear operator} 
\medskip
\noindent

Among linear differential operators, some of the simplest are the monomial operators of the form $T_\ell=z^{\ell}\frac{d}{dz}$ for some $\ell\in \bZ$. Setting $\omega=z^{-\ell} dz$, these operators fall within the scope of the theory. Moreover, since $\omega=d\phi$ for $\phi(z)=\frac{z^{1-\ell}}{1-\ell}$, the study of the asymptotic of the zeros of the iterates of $T_\ell$ reduces to the classical problem involving the usual derivative with respect to the variable $\phi$. 
In fact this procedure corresponds to taking the appropriate Riemann surface defined by $\phi$, its branched covering of $\bC P^1$ and using Theorem~\ref{thm:MAIN}.

\medskip
\noindent {\bf Root asymptotic for a simple example.}

For any $\ell \geq 2$, we   calculate the root asymptotic  of the sequence $\{T_\ell^n\left(-\frac{1}{z+a}\right)\}$ with $a\neq 0$ when $n\to +\infty$. 
\par
Theorem~\ref{thm:MAIN} in the case under consideration claims the following. The principal polar locus $\mathcal{PPL}(\omega,\frac{1}{z+a})$ is $\lbrace{-a,\infty \rbrace}$. The image of the Voronoi diagram under $\phi$ is the midline of $[\phi(\infty),\phi(-a)]$. It is a straight line $L$ parametrised by $t \mapsto (\frac{1}{2}+it)\phi(-a)$ (where $t \in \mathbb{R}$).
\par
The inverse image of line $L$ under ramified cover $\phi$ is a planar algebraic curve of equation $\Re(\frac{z}{a})^{1-\ell}=\frac{(-1)^{\ell-1}}{2}$. It is a curve formed by $\ell -1$ loops attached to $0$ each contained in a cone of angle $\frac{2\pi}{\ell-1}$. The limit set $\mathcal{L}(T_{\ell},\frac{1}{z+a})$ is the branch of the curve contained in the cone containing $-a$. An illustration of the above root distribution can be found in Fig.~\ref{vvN2}. (The explicit formula for the lemniscate is obtained from the above equation by using $\ell=4$ and $a=1$).

\begin{figure}[!hbt]
\centering
\includegraphics[width=0.6\linewidth]{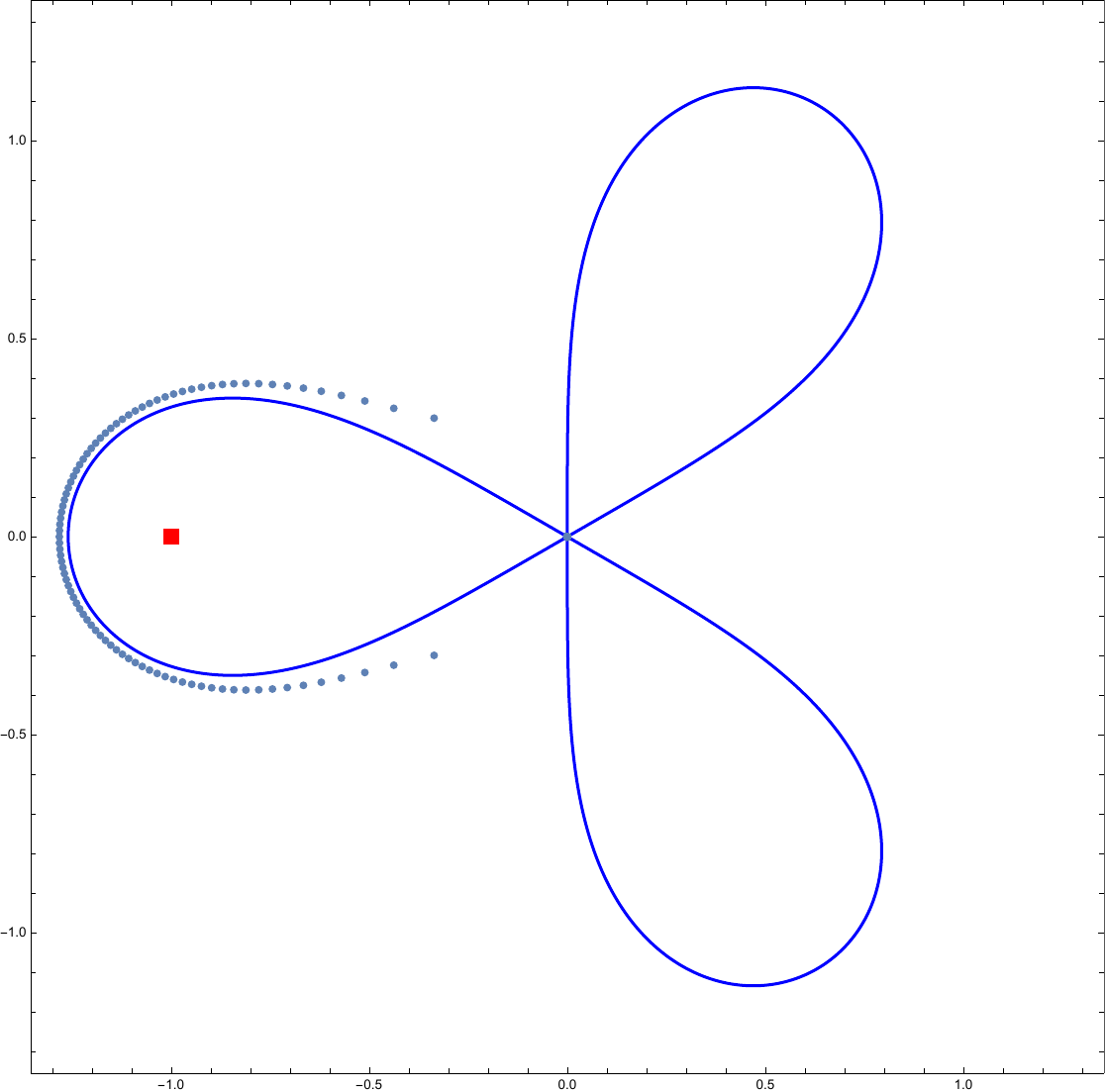}
\caption {The lemniscate with the roots of $T_4^{90}\left(-\frac{1}{z+1}\right)$ accumulating on one of its ovals.}\label{vvN2}
\end{figure}


\subsection{Two final examples in genus one.}

In the complex plane $\mathbb{C}$, we consider the lattice $\Lambda =\mathbb{Z} + \tau\mathbb{Z}$, where $\tau = \exp(\pi i/3)$. Let $X = \mathbb{C}/\Lambda$ be the associated elliptic curve. 

\begin{enumerate}
    \item Let $\omega = dz$  and let $ \wp(z)$ be the Weierstra{\ss} elliptic function on $X$, where $z$ is the coordinate on $\mathbb{C}$. Since $f$ has one pole of order 2 at 0 and $\omega$ has neither zeros nor poles, the principal polar locus $\mathcal{PPL}(\omega, \wp)$ in $X$ is $\{0\}$. Using the primitive $\phi(z) = z$ of $\omega$, it is seen that the limit set $\mathcal{L}(T_{\omega}, \wp)$ in  $X$ is the quotient by $\Lambda$ of the usual Voronoi diagram on $\mathbb{C}$ determined by the lattice points, see Figure \ref{fig:torus1}.

    \item For a more complicated example, we consider the derivative of the Weierstra{\ss} elliptic function on $X$, which we denote by $\wp_z$. Let $\omega = \wp_z(z)dz$ be the meromorphic 1-form and $f(z) = \wp_z(z)$ be the meromorphic function  on $X$. It is a standard fact that $\frac{1}{2}, \frac{\tau}{2}, \frac{1+\tau}{2}$ are the three simple zeros of $\omega$ in $X$. These points become poles of $T_\omega^n(f)$ for all $n \geq 1$, so $f$ is not locally factorised by a primitive of $\omega$ there by Lemma \ref{lem:criticalfactorisation}. Hence, they belong to $\mathcal{PPL}(\omega, f)$. Moreover, since the poles of $\omega$ and $f$ are the same, we have that $\mathcal{PPL}(\omega, f)$ is $\{\frac{1}{2}, \frac{\tau}{2}, \frac{1+\tau}{2} \}$. 
    
    The limit set $\mathcal{L}(T_{\omega}, f)$ in $X$ is the preimage under the ramified cover $\wp$ of the Voronoi diagram in the $\mathbb{C}$-plane determined by $\wp(\mathcal{PPL}(\omega, f))$, where the ramification points are precisely the points in  $\mathcal{PPL}(\omega, f)$, see Figure  \ref{fig:torus2}.
    \end{enumerate}

\begin{center}
    \begin{figure}[h]
        \centering
        \includegraphics[width=0.45\linewidth]{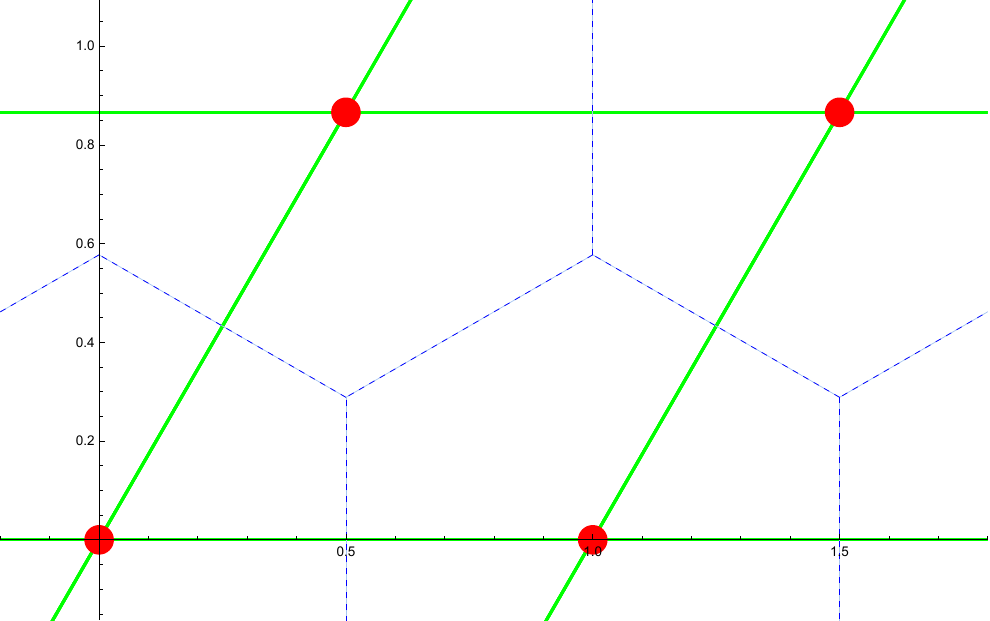} \quad
         \includegraphics[width=0.45\linewidth]{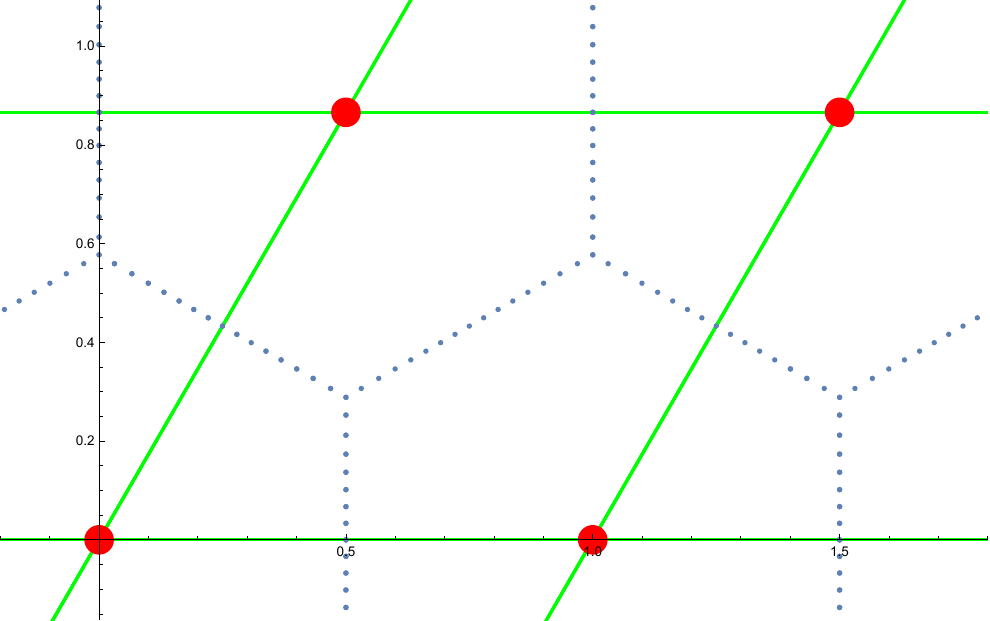} 
        \caption{Let $\Lambda = \mathbb{Z} + e^{\frac{i\pi}{3}}\mathbb{Z}$ and $\omega = dz$. Left: The Voronoi diagram determined by the poles of $\wp(z)$ (red dots) consists of the blue dashed lines. Right: The zeros of $\wp^{(45)}$ are in blue dots. In this case, parallelograms cut out by the green lines are fundamental domains of the action of $\Lambda$ on $\mathbb{C}$.}
        \label{fig:torus1}
    \end{figure}
\end{center}

\begin{center}
    \begin{figure}[h]
        \centering
        \includegraphics[width=0.45\linewidth]{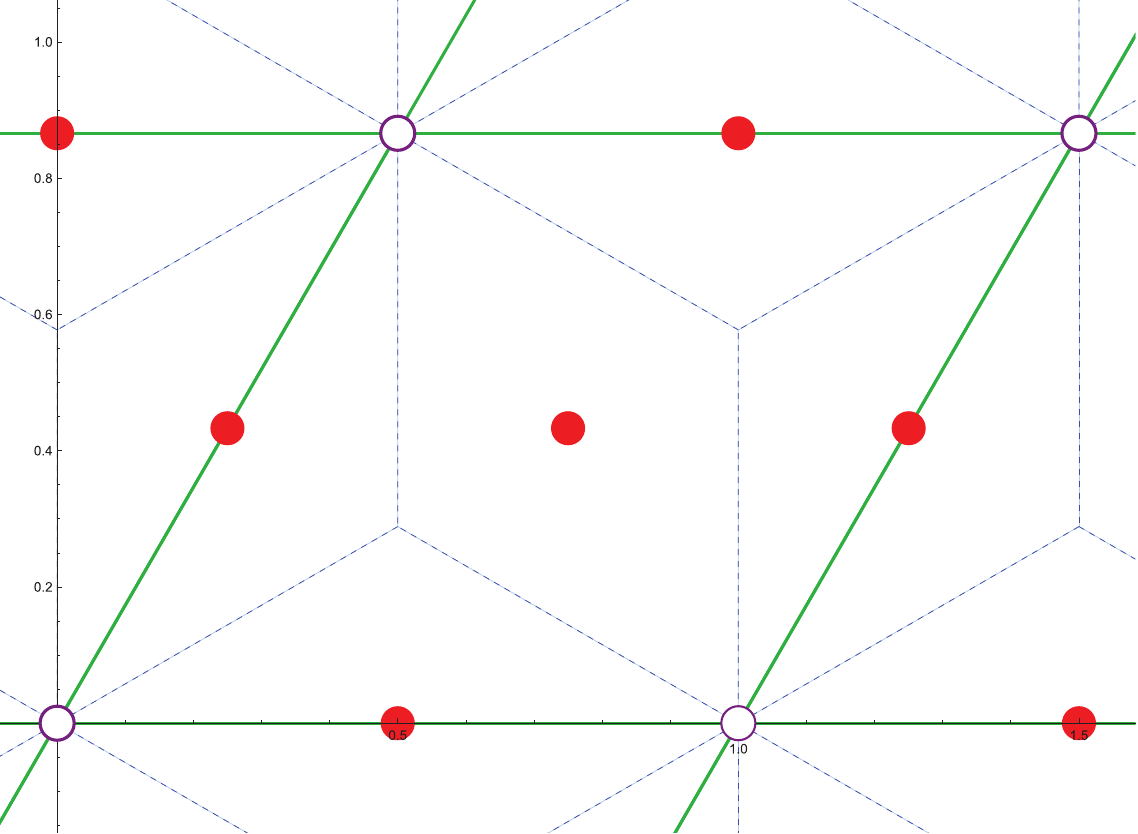}\quad
        \includegraphics[width=0.45\linewidth]{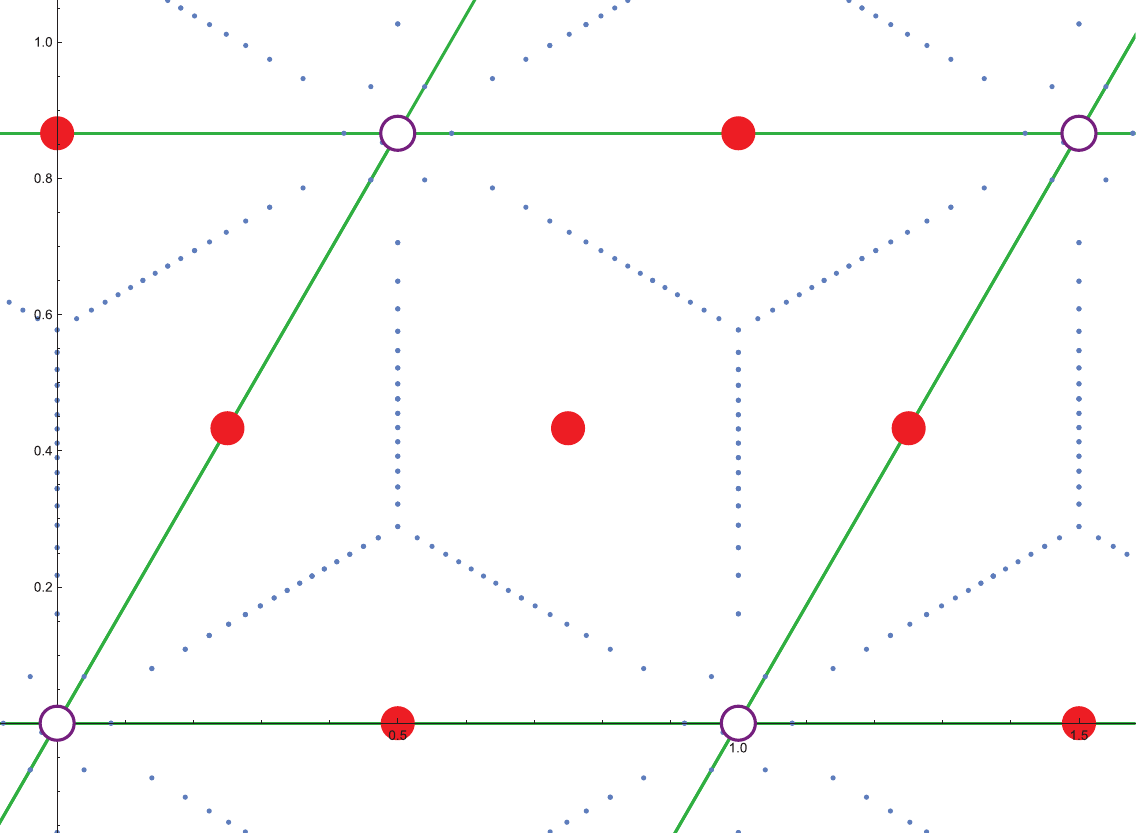}
        \caption{Let $\Lambda = \mathbb{Z} + e^{\frac{i \pi}{3}}\mathbb{Z}$. Left: The red dots are the points of $\mathcal{PPL}(\omega, f)$, where $\omega = \wp_z(z)dz$ and $f(z) = \wp_z(z)$, which determine the Voronoi diagram in blue dashed lines. The green lines cut out fundamental domains of the action of $\Lambda$ on $\mathbb{C}$. The white circles in the corners  are the poles of $\omega$. Right: The dark blue dots are the zeros of $T_{\omega}^{25}(f)$ in $\mathbb{C}$.}
        \label{fig:torus2}
    \end{figure}
\end{center}

\medskip
\noindent {\bf Geometric reconstruction of the translation surface.}

In what follows, we show how asymptotic root-counting provides a way to construct explicitly the translation surface $(X,\wp_z(z)dz)$ by gluing Euclidean polygons.
\par
We have deliberately chosen $f(z)=\wp_z(z)$ to have its poles at the poles of $\omega$ so that the principal polar locus $\mathcal{PPL}(\omega, f)$ depends only on $\omega$ (and hence only on the translation structure).
\par
The Voronoi diagram admits a dual graph whose edges, connecting the zeros of $\omega$
(the conical singularities of the flat metric $|\omega|$), are realized by saddle connections, thereby forming the Delaunay decomposition of the translation surface.
\par
In Figure~\ref{fig:placeholder}, these Delaunay edges decompose a fundamental parallelogram into three pieces: two triangles called $A,B$ (drawn between the three zeros of $\omega$) and a hexagon called $C$ (containing the triple pole of $\omega$).
\par
The lengths and slopes of the sides of these polygons in the flat coordinates can be explicitly computed by integrating $\omega$ along the corresponding homology classes. We obtain that $A$ and $B$ are equilateral triangles while $C$ is the double cover of the complement of an equilateral triangle in the infinite flat plane (the six interior angles of the hexagon are equal to $\frac{5\pi}{3}$). The sides of these polygons are then glued according to the pattern prescribed by the dual graph of the Voronoi diagram. In particular, the three zeros correspond to conical singularities of angle $4\pi$. Our flat surface is now constructed!

\begin{center}
\begin{figure}[h]
    \centering
    \includegraphics[width=0.75\linewidth]{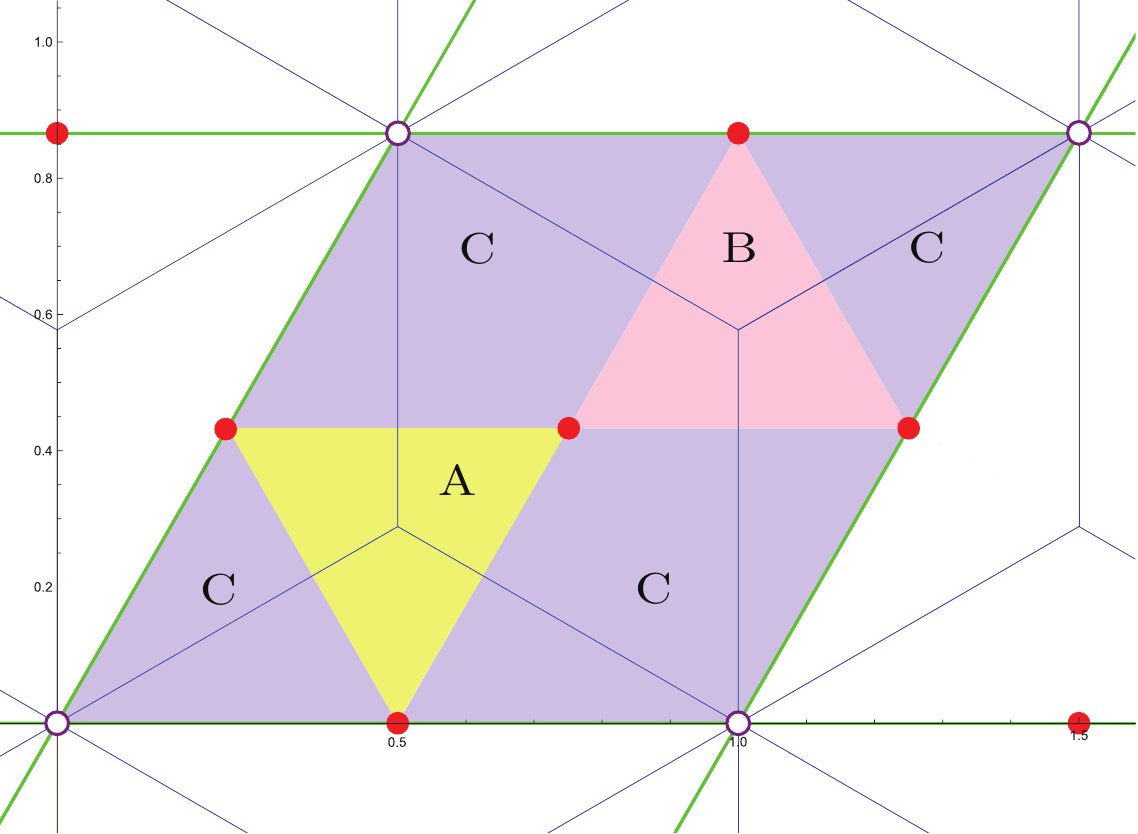}
    \caption{In a fundamental parallelogram, we draw the Delaunay cells dual to the Voronoi diagram (blue line) given in Figure \ref{fig:torus2}. Observe that $(X,\omega)$ decomposes into two equilateral triangle faces $A$ and $B$ and one hexagonal face $C$ centered at the pole of $\omega$. }
    \label{fig:placeholder}
\end{figure}
\end{center}

\subsection{Organisation of the paper}

\begin{itemize}
    \item In Section~\ref{sec:Translation}, we provide background on translation surfaces. 
    \item In Section~\ref{sec:Geometry}, we introduce the principal polar locus and give the analytic construction of Voronoi diagrams in translation surfaces.
    \item In Section~\ref{sec:ZeroFree}, we characterise asymptotically zero-free regions in terms of Voronoi diagrams.
    \item In Section~\ref{sec:Edges}, we describe the asymptotic distribution of zeros of $T_{\omega}^{n} f$ on Voronoi edges and prove Theorem~\ref{thm:MAIN}.
    \item In Section~\ref{sec:algfunc}, we generalise previous results of Prather-Shaw \cite{PrSh2} on sequences of derivatives of algebraic functions using the approach of the present paper. 
    \item In Section~\ref{sec:Outlook}, we outline several areas of application for the obtained  results, as well as potential generalisations.
\end{itemize}

\medskip \noindent
\emph {Acknowledgements.} The second author wants to acknowledge the financial support of his research provided by the Swedish Research Council grant 2021-04900. He is sincerely grateful to Beijing Institute for Mathematical Sciences and Applications for the hospitality in Fall 2023 where the major part of this research has been carried out. Research by the third author is supported by the Beijing Natural Science Foundation (Grant IS23005) and the French National Research Agency under the project TIGerS (ANR-24-CE40-3604). The authors want to thank Vincent Delecroix, Alexandre Eremenko, Yi Huang, Pavel Kurasov, Aud Lundholm, Curtis McMullen, Dmitry Novikov, Anton Zorich and the anonymous referees for their interest, valuable remarks and discussions.

\section{Primer on translation surfaces}\label{sec:Translation}

\subsection{Translation structures}\label{sub:translation}

Any non-zero meromorphic $1$-form $\omega$ on a (possibly open) Riemann surface $X$ defines on $X$ a geometric structure called \textcolor{blue}{\textit{translational}}. Denote by
\begin{itemize}
    \item $X^{\ast}$ the surface $X$ punctured at the poles of $\omega$;
    \item $X^{\ast\ast}$ the surface $X$ punctured at the zeros and the poles of $\omega$.
\end{itemize}
Local primitives of $\omega$ are locally injective on $X^{\ast\ast}$. They form an atlas of local biholomorphisms from $X^{\ast\ast}$ to $\mathbb{C}$. Additionally, since any two local primitives of the same differential $1$-form $\omega$ differ by a constant, transition maps between two distinct charts of the atlas are \textcolor{blue}{\textit{translations}} of the complex plane. Therefore, we say that $\omega$ endows the Riemann surface $X^{\ast\ast}$ with a \textcolor{blue}{\textit{translation structure}}. The pair $(X,\omega)$ is called a \textcolor{blue}{\textit{translation surface}}.
\par
As a geometric structure, a translation structure can be presented as a pair $(\text{dev},\chi)$ where $\text{dev}$ (for developing map) is a multivalued primitive on $X^{\ast\ast}$ (or equivalently a primitive of the lift $\omega$ to the universal cover of $X^{\ast\ast}$) while $\chi$ is a holonomy representation of $\pi^{1}(X^{\ast\ast})$ where every closed loop $\gamma$ in $X^{\ast\ast}$ is mapped to the translation of vector $\int_{\gamma} \omega$. Indeed, two primitives of $\omega$ defined by analytic continuation along two distinct paths differ by the period of $\omega$ on the loop formed the concatenation of the two paths.
\par
In the complement $X^{\ast\ast}$ of the zeros and the poles of $\omega$, charts provide a local isomorphism with the standard complex plane $(\mathbb{C},dz)$. In the following sections, we introduce local models for neighbourhoods of zeros and poles of a meromorphic $1$-form.

\begin{figure}
    \centering
    \includegraphics[width=0.85\linewidth]{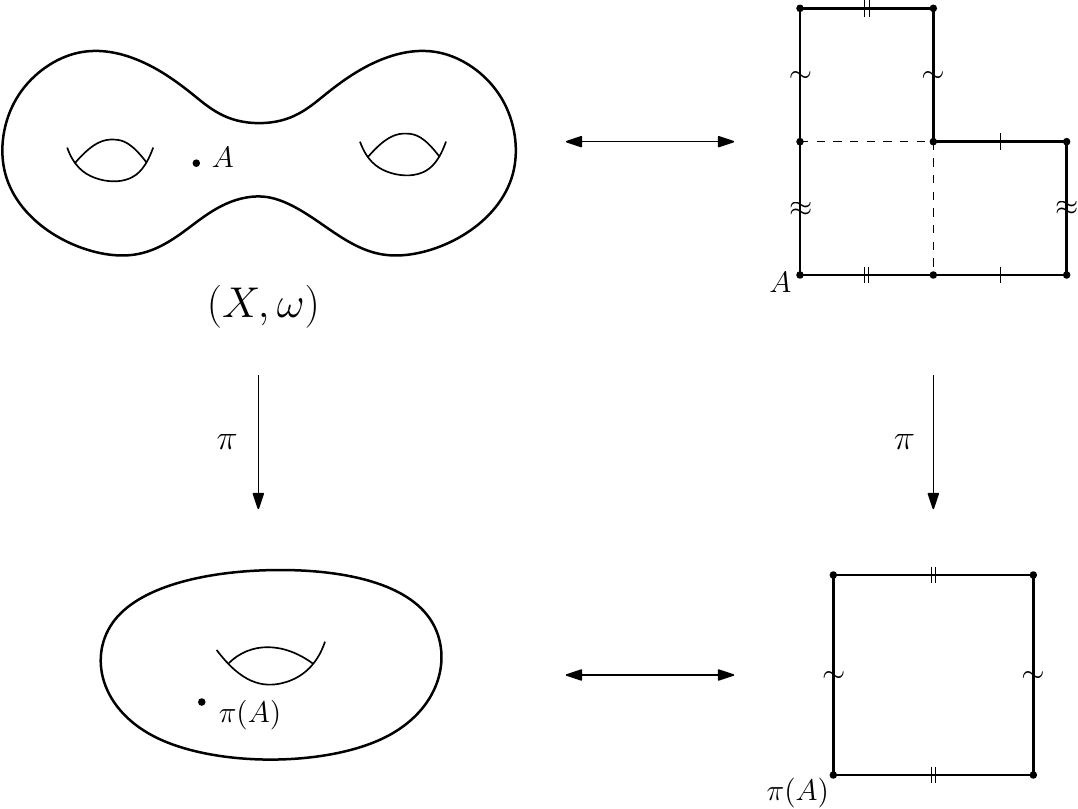}
    \caption{Top: The L-shaped translation surface $(X,\omega)$, a genus-two surface with a double zero $A$ and its flat model constructed by gluing three unit squares together (vertex $A$ is a conical singularity of angle $6\pi$).\\ Bottom: Projection $\pi$ from $X$ to the square torus with a unique ramification point.}
    \label{fig:LShaped}
\end{figure}

\subsection{Local models of zeros}\label{sub:LocalModelsZeros}

At a zero of order $k$, a meromorphic $1$-form $\omega$ locally writes as $z^{k}dz$ for some local coordinate $z$. Then, $\omega$ locally writes as $d\phi$ where $\phi=\frac{1}{k+1}z^{k+1}$. As $\phi$ locally defines a cover of degree $k+1$ the disk ramified over $0$, a local model for a punctured neighbourhood of zero of order $k$ is a cyclic cover of degree $k+1$ of the punctured disk.
\par
A classical example is the L-shaped surface $(X,\omega)$, defined as a cover $\pi$ of degree $3$ of the standard square torus $(\mathbb{C}/\Gamma,dz)$ where $\Gamma=\mathbb{Z}+i\mathbb{Z}$. The cover $\pi$ has a unique ramification point of order $3$ so $\omega$ has a double zero whose local model is a cyclic cover of degree $3$ of the punctured disk, see Figure~\ref{fig:LShaped}.

\subsection{Geodesics and saddle connections}\label{sub:GeodesicSaddle}

Since translations of the plane map straight lines to straight lines of the same slope, translation surfaces admit a well-defined notion of straight line, which we will call (unparameterized) \textcolor{blue}{geodesics}. Such arcs in $X^{\ast\ast}$ are mapped to straight segments in the charts of the translation atlas. Moreover, these geodesics carry a well-defined slope in $\mathbb{RP}^{1}$.
\par
A remarkable consequence of the fact that geodesics on translation surfaces preserve a constant slope is that two geodesics of the same slope cannot intersect. For the same reason, no geodesic can intersect itself. Therefore, for each slope $\theta \in \mathbb{RP}^{1}$, $X^{\ast\ast}$ admits a directional foliation, whose dynamical properties have been the subject of extensive study (see, for instance, \cite{Zor}).

\begin{definition}
In a translation surface $(X,\omega)$, we call \textcolor{blue}{\textit{saddle connection}} an arc $\gamma:[0,1]\to X^{\ast}$ such that:
\begin{itemize}
    \item arc $\gamma(]0,1[)$ is a geodesic of $X^{\ast\ast}$ (in particular it does not contain any zero of $\omega$);
    \item endpoints $\gamma(0)$ and $\gamma(1)$ are zeros of $\omega$ (possibly the same). 
\end{itemize}
\end{definition}

Saddle connections, such as those shown dotted in Figure~\ref{fig:LShaped}, provide a convenient way to decompose a translation surface into simpler (polygonal) pieces.
\par
For a saddle connection $\gamma$, the period $\int_{\gamma} \omega$ has a straightforward geometric interpretation: its argument is the slope of $\gamma$ as an oriented geodesic while its modulus is its length with respect to the singular flat metric $|\omega|$ as explained in Section~\ref{sub:FlatMetric}.
\par
The period of a saddle connection depends only on its relative homology class in the homology group $H_{1}(X^{\ast},Z_{\omega})$ of the surface punctured at the poles relatively to the zeros.
\par
Understanding which relative homology classes in $H_{1}(X^{\ast},Z_{\omega})$ are actually realized by the saddle connections of $(X,\omega)$ is a delicate problem, closely related to the difficulty of explicitly constructing a translation surface by gluing polygons along saddle connections.

\subsection{Local models of simple poles}\label{sub:LocalModelsSimplePoles}

At a simple pole, a meromorphic $1$-form $\omega$ locally writes as $\frac{\lambda dz}{z}$ where $\lambda \in \mathbb{C}^{\ast}$ (up to a biholomorphic change of variable). Local primitives of $\omega$ are of the form $\lambda \ln(z)$. 
\par
In a \textcolor{blue}{\textit{translation cylinder}} constructed by the quotient $\mathbb{C}/2i\pi\lambda\mathbb{Z}$, the standard differential form $dz$ extends to the two ends as simple poles of residues $\lambda$ and $-\lambda$. Multivalued chart $\phi(z)= \lambda \ln(z)$ identifies a neighborhood of any simple pole with an end of such a translation cylinder.
\par
A (punctured) neighborhood of such a simple pole is then foliated by a one-parameter family of closed geodesics of period $2i\pi\lambda$. As they have same period, these closed geodesics are parallel in the flat charts.

\subsection{Local models of poles of higher order}\label{sub:LocalModelsPoles}

A systematic construction of local models of poles of order $k \geq 2$ is given in \cite{Boissy}; we present them here without proof. First note that up to biholomorphic change of variable, at a pole of order $k \geq 2$, a meromorphic $1$-form writes as $\left(\frac{1}{z^{k}}+\frac{\lambda}{z} \right)dz$ for some $\lambda \in \mathbb{C}$ (see for example Proposition~3.1 in \cite{BCGGM}).
\par
We first discuss the case $k=2$. At a double pole with zero residue $(\lambda=0)$, the translation structure is isomorphic to the point at infinity in $\mathbb{CP}^{1}$ compactification of the usual flat plane $(\mathbb{C},dz)$.
\par
To obtain a local model for a double pole with a non-zero residue, we proceed as follows. From the infinite flat plane, we remove a semi-infinite rectangle defined by the inequalities $\Re(z) \geq 0$ and $0 \leq \Im(z) \leq 2\pi$. Identifying the two boundary half-lines produced by translation $z \mapsto z+2i\pi$, the singularity produced at infinity is a double pole with a residue equal to $\lambda =1$. Multiplying the differential by some complex factor (this amounts to scale and rotate the semi-infinite rectangle), we can obtain every possible non-zero value for $\lambda$.
\par
Any pole of order $k \geq 3$ and residue $\lambda$ can be constructed by a cover of degree $k-1$ ramified over a double pole of residue $\frac{\lambda}{k-1}$. In particular, a pole of order $k$ with zero residue can be identified with the point at infinity in an infinite cone of angle $(2k-2)\pi$ (in other words, a cover of degree $k-1$ of the flat plane ramified over one point).

\subsection{Flat metric}\label{sub:FlatMetric}

Just as translations preserve straight lines and slopes, they also preserve the length of segments. Hence, a translation structure determines a metric.
\par
On a Riemann surface $X$, the differential $\omega$ induces a flat metric $|\omega|$ in the following way: for any point $z_{0} \in X^{\ast\ast}$, the distance between $z_{0}$ and any other point $z$ in a small enough neighbourhood of $z_{0}$ is defined as 
$\lvert \int\limits_{z_{0}}^{z} \omega \rvert$. For example, the standard differential $dz$ endows the complex plane $\mathbb{C}$ with the standard Euclidean metric $|dz|$. Any chart $\phi$ of the translation atlas, i.e. a local primitive of $\omega$, conjugates the standard differential $dz$ in the complex plane with a differential $\omega$ defined on $X$. In other words, we have $\phi^{\ast}dz = \omega$. A translation surface is formed by pieces of the standard complex plane $\mathbb{C}$ glued together with the help of translations which are in particular isometries.
\par
Flat metric $|\omega|$ defined on the punctured surface $X^{\ast\ast}$ extends naturally to zeros of $\omega$ in the following way. As explained in Section~\ref{sub:LocalModelsZeros}, a punctured neighborhood of a zero of order $k$ of $\omega$ is locally isomorphic to a cyclic cover of degree $k+1$ of the punctured disk. It follows that the flat metric $|\omega|$ extends to such a zero as a conical singularity of angle $2(k+1)\pi$. Therefore the punctured surface $X^{\ast}$ is an Euclidean surface with conical singularities, see \cite{Tr}. Figure~\ref{fig:LShaped} provides an explicit example of a flat surface with a conical singularity of angle $6\pi$ corresponding to a double zero of the differential.

\subsection{Translation surfaces are metrically complete}

Remarkably, as soon as a meromorphic differential defined on a compact Riemann surface has poles, the metric structure it defines on the surface punctured at these poles is no longer compact. The following fact is well known (see \cite{Boissy} for a reference on the flat geometry of meromorphic $1$-forms)

\begin{lemma}\label{lem:complete}
Let $X$ be a compact Riemann surface with a non-zero meromorphic $1$-form $\omega$. Then the punctured surface $X^{\ast}$ is a complete metric space for the singular flat metric $|\omega|$.
\end{lemma}

\begin{proof}
For any pair of points $x,y \in X^{\ast}$, $d(x,y)$ is the distance between $x,y$ for the singular flat metric $|\omega|$. We have to prove that any Cauchy sequence $(z_{n})_{n \in \mathbb{N}}$ of points in $X^{\ast}$ converges to some limit point in $X^{\ast}$. Since $X^{\ast}$ is locally compact and $\omega$ has finitely many poles, the question reduces to the only case when  $(z_{n})_{n \in \mathbb{N}}$ converges in $X$ to some point $z_{\infty}$ which is a pole of $\omega$. We will prove that no such sequence can be a Cauchy sequence in the flat metric $|\omega|$.
\par
We first assume that $z_{\infty}$ is a simple pole of $\omega$. Then, in a neighbourhood $V$ of $z_{\infty}$, up to biholomorphic local change of variable, we can assume that $z_{\infty}=0$ and $\omega=\frac{\lambda dz}{z}$ for some $\lambda \in \mathbb{C}^{\ast}$, see Section~\ref{sub:LocalModelsSimplePoles}. Up to a choice of a subsequence of $(z_{n})_{n \in \mathbb{N}}$, we assume that for any $n \in \mathbb{N}$, $z_{n} \in V \setminus \lbrace{ 0 \rbrace}$. Then for any path $\gamma$ joining two points $z_{n},z_{m}$ in the sequence, we have $\vert \int\limits_{\gamma} \omega \vert \geq |\lambda|\cdot \vert \ln(|z_{m}|/|z_{n}|) \vert$. It follows that for a fixed $m$, $d(z_{n},z_{m})$ can be made arbitrarily large as $n \to + \infty$. In other words, $(z_{n})_{n \in \mathbb{N}}$ cannot be a Cauchy sequence.
\par
Now we consider the case where $z_{\infty}$ is a pole of order $k \geq 2$ of $\omega$. Up to a biholomorphic local change of variable in a neighbourhood $V$ of $z_{\infty}$, we can assume that $z_{\infty}=0$ and normalise $\omega$ as $( \frac{\lambda}{z} + \frac{1}{z^{k}})dz$ for some $\lambda \in \mathbb{C}$, see Section~\ref{sub:LocalModelsPoles}.
\par
Then, $\omega$ admits a (multi-valued) local primitive $\phi: z \mapsto \lambda \ln (z) - \frac{1}{(k-1)}z^{1-k}$. We deduce that for a simple path $\gamma$ joining a pair of points $z_{n},z_{m}$ in the sequence such that $|z_{m}|>|z_{n}|$, we have 
$$\vert \int\limits_{\gamma} \omega \vert \geq
\frac{|z_{n}|^{1-k} - |z_{m}|^{1-k}}{k-1}- |\lambda|\cdot \vert \ln(|z_{m}|/|z_{n}|) \vert - 2\pi |\lambda|.$$
The term $2\pi |\lambda|$ takes into account the fact that there are two homotopy classes of simple paths joining $z_{n}$ to $z_{m}$ in $V$ and the corresponding periods differ by the period $2i\pi\lambda$ of $\omega$ at the pole (where $\lambda$ is the residue at the pole).
\par
Since $|z_{n}|^{1-k}$ grows much faster than $|\ln(|z_{n}|)|$ as $z_{n} \to 0$, we deduce that for a fixed $m$, $d(z_{n},z_{m})$ can be made arbitrarily large as $n \to + \infty$. The sequence $(z_{n})_{n \in \mathbb{N}}$ is not a Cauchy sequence.
\end{proof}

\section{Principal polar locus and Voronoi functions}\label{sec:Geometry}

\subsection{Growth of the order of poles under iterations of $T_{\omega}$}\label{sub:Growth}

Notice that unless a function $f$ is very special, one expects it to develop poles at the zeros of $\omega$ under iterative applications of the operator $T_{\omega}$. The following lemma describes this phenomenon explicitly.

\begin{lemma}\label{lem:criticalfactorisation}
Consider a non-zero meromorphic $1$-form $\omega$ and a meromorphic function $f$ on a Riemann surface $X$. Then for any point $z_{0} \in X$ which is not a pole of $\omega$ the following statements are equivalent:
\begin{itemize}
    \item no function in the sequence $(f_{n})_{n\in \mathbb{N}}$  has a pole at $z_{0}$;
    \item $f$ is locally factorised at $z_{0}$ by a primitive of $\omega$ (see Definition~\ref{defn:localfactor}).
\end{itemize}
\end{lemma}

\begin{proof}
Up to local biholomorphic change of variable $z$, we can assume that $z_{0}=0$ and that $\omega=z^{m}dz$ for some $m \in \mathbb{N}$. An arbitrary locally defined holomorphic function $f$ can be written as a power series in  $z$, i.e.,  $$f(z)= \sum\limits_{k=0}^{+\infty} a_{k}z^{k}.$$
Then $T_{\omega} (f)$ is given by a Laurent series of the form:
$$T_{\omega} (f) = \sum\limits_{k=0}^{+\infty} a_{k}kz^{k-m-1}.$$
It follows immediately that if no function in the sequence $(f_n)_{n\in N}$ has a pole at $0$, then $a_{k}=0$ for all $k \notin (m+1)\mathbb{Z}$.\newline
Introducing the local primitive $\phi=\frac{z^{m+1}}{m+1}$ of $\omega$, we obtain
$$f(z)= \sum\limits_{k=0}^{+\infty} a_{k(m+1)}(m+1)^{k}\phi^{k}.$$
Therefore, in a neighbourhood of $0$, $f$ factorises as $g \circ \phi$ where $g$ is a holomorphic function  given by the above series and well-defined near $0$.
\par
Conversely, for any meromorphic function $f$ which locally factorises as $g \circ \phi$, direct computation proves that for any $k$, we have:
$$
T_{\omega}^{k} (f) = g^{(k)} \circ \phi.
$$
Therefore no iterate $T_{\omega}^{k} (f)$ of $f$ has a pole at $z_{0}$.
\end{proof}

The following statement establishes a dichotomy between the points belonging to the principal polar locus $\mathcal{PPL}(\omega,f)$ and the points outside it (see Definition~\ref{defn:PPL}). On the principal polar locus, the function  $T^k_{\omega}(f)$ has poles whose orders grow linearly with $k$ while the sum of orders of the poles outside $\mathcal{PPL}(\omega,f)$ remains bounded. In particular, we obtain that when the sum of orders of poles (and thus the sum of orders of zeros) grows to infinity when $k\to +\infty$, the principal polar locus must be non-empty. We prove that the sum of orders of poles has linear growth and provide a formula for the leading coefficient of the latter linear dependence.

\begin{proposition}\label{prop:growth}
Consider a non-zero meromorphic $1$-form $\omega$ and a meromorphic function $f$ on a compact Riemann surface $X$. Then there is an integer $M>0$ such that:
\begin{itemize}
    \item for any $k>M$ and any point $p\in \mathcal{PPL}(\omega,f)$ that is a zero of order $a_{p}>0$ of $\omega$, $p$ is a pole of $T_{\omega}^{k}(f)$ of order $\alpha_{p}+k(a_{p}+1)>0$, where $\alpha_{p}$ is some constant;
    \item for any $k>M$ and any point $p\in \mathcal{PPL}(\omega,f)$ that is a regular point of $\omega$, $p$ is a pole of $T_{\omega}^{k}(f)$ of order $\alpha_{p}+k$, where $\alpha_{p}$ is some constant;
    \item for any $k > M$, the sum of orders of the poles of $T_{\omega}^{k}(f)$ outside $\mathcal{PPL}(\omega,f)$ is constant. All these poles are
    located at the simple poles of $\omega$.
\end{itemize}
\end{proposition}

\begin{proof}
Iterating $T_{\omega}$ we get that for any $k \in \mathbb{N}$, any pole of $T_{\omega}^{k}(f)$ is either a pole of $f$ or a zero of $\omega$. Since $X$ is compact, we have only finitely many such points to examine.\newline
Let us first consider the case of a point $p$ which is a pole of $f$ of order $m$ and a pole of $\omega$ of order $s_{p}$. Direct computation shows that the function $T_\omega(f)$ has at $p$ a singularity of order $-m +s_{p}-1$. Therefore, if $p$ is a simple pole of $\omega$, it still remains a pole  of $T_{\omega}^{k}(f)$ of order $m$ for any $k \in \mathbb{N}$. On the other hand, if $p$ is a pole of $\omega$ of order at least two, then $T_{\omega}^{k}(f)$ is holomorphic at $p$ provided that $k$ is large enough. In both cases, $p$ does not belong to $\mathcal{PPL}(\omega,f)$.\newline
Now, let us consider the case when $p$ is a zero of $\omega$ and $f$ is locally factorised by a primitive of $\omega$. If $p$ is not a pole of $f$, then Lemma~\ref{lem:criticalfactorisation} proves that $p$ is not a pole of any function in the sequence $(f_n)_{n\in \mathbb{N}}$. Thus, we have shown that after finitely many steps, the sum of orders of the poles of $T_{\omega}^{k}(f)$ outside $\mathcal{PPL}(\omega,f)$ stabilises to the sum of  orders of the poles of $f$ coinciding with the simple poles of $\omega$.\newline
Lemma~\ref{lem:criticalfactorisation} implies that for any point $p$ of $\mathcal{PPL}(\omega,f)$,  there exists $k_{p} \in \mathbb{N}$ such that $T_{\omega}^{k_{p}} (f)$ has a pole at $p$. Then direct computation proves that the order of the pole increases by $a_{p}+1$ with each application of $T_{\omega}$ which finishes the proof.
\end{proof}

\begin{corollary}\label{cor:growthcoefficient}
Consider a non-zero meromorphic $1$-form $\omega$ and a meromorphic function $f$ on a compact Riemann surface $X$ such that $\mathcal{PPL}(\omega,f) \neq \emptyset$. For any $z \in \mathcal{PPL}(\omega,f)$, let $a_{z} \geq 0$ be the order of $\omega$ at $z$ (poles of $\omega$ do not belong to the principal polar locus).
\par
Then the sum $\mathcal{Z}_{n}$ of orders  of the poles of the meromorphic function $T_{\omega}^{n}(f)$ (and therefore the sum of orders of its zeros) grows when $n \to +\infty$ as $\mathcal{Z}_{n} \sim A\cdot n$  where
$$
A = \sum\limits_{z \in \mathcal{PPL}(\omega,f)} (a_{z}+1).
$$
\end{corollary}

\subsection{Voronoi functions}\label{sub:VorFunctions}

For a meromorphic function $f$ in $\mathbb{C}$, the Voronoi diagram of its set of poles can be defined in terms of maximal embedded disks  disjoint from these poles. \textcolor{blue}{\textit{Voronoi functions}} play a similar role on translation surfaces.

\begin{definition}\label{defn:VorFunc}
Consider a compact Riemann surface $X$  with a non-zero meromorphic $1$-form $\omega$ and a meromorphic function $f$ on $X$. Let $z \in X^{\ast} \setminus \mathcal{PPL}(\omega,f)$.
\par
Denote by $\phi$ a local primitive of $\omega$ such that  $\phi(z)=0$.  The \textcolor{blue}{\textit{Voronoi function}} $g_{z}$ is the holomorphic function defined in a neighbourhood of the origin in the complex plane $\mathbb{C}$ and satisfying $g_{z} \circ \phi =f$ (see Definitions~\ref{defn:localfactor} and~\ref{defn:PPL} for the existence and uniqueness of $g_{z}$ if $z$ is a zero of $\omega$).
\par
The \textcolor{blue}{\textit{critical radius} }$\rho(z)$ is the radius of convergence of $g_{z}$ at $0$. Immediate computation shows that for any $k \geq 1$, $g_{z}^{(k)} \circ \phi = T_{\omega}^{k}(f)$.
\end{definition}

Our next result shows that the behaviour of $g_{z}$ on the boundary of its disk of convergence at $0$ reflects the geometry of the translation surface $(X,\omega)$.

\begin{lemma}\label{lem:path}
Consider a compact Riemann surface $X$ with a non-zero meromorphic $1$-form $\omega$ and a meromorphic function $f$. Let $z \in X^{\ast} \setminus \mathcal{PPL}(\omega,f)$, and denote by $\tilde{z}$ some preimage of $z$ in the universal cover $\pi:\tilde{X^{\ast}}\rightarrow X^{\ast}$. Consider the primitive $\tilde{\phi}$ of $\pi^{\ast}\omega$ in $\tilde{X^{\ast}}$ such that $\tilde{\phi}(z)=0$.
\par
Then for any path $\sigma:[0,1] \rightarrow \mathbb{C}$ such that 
\begin{itemize}
    \item $\sigma(0)=0$;
    \item $|\sigma(s)|< \rho(z)$ for any $s \in [0,1[$;
    \item $|\sigma(1)|\leq \rho(z)$;
\end{itemize}
there exists a path $\tilde{\gamma}:[0,1] \rightarrow \tilde{X^{\ast}}$ such that  $\tilde{\phi} \circ \tilde{\gamma}=\sigma$. In addition one of the following two statements holds:
\begin{enumerate}

    \item either $\pi \circ \tilde{\gamma}(1) \notin \mathcal{PPL}(\omega,f)$ and $g_{z}$ extends to $\sigma(1)$ as a holomorphic function;
    
    \item or $\pi \circ \tilde{\gamma}(1) \in \mathcal{PPL}(\omega,f)$ and $g_{z}$ does not extend to $\sigma(1)$ as a holomorphic function. However, it extends in a neighbourhood of $\sigma(1)$ as a convergent Puiseux series.
\end{enumerate}
\end{lemma}

\begin{proof}
The existence of a path $\tilde{\gamma}:[0,1] \rightarrow \tilde{X^{\ast}}$ satisfying $\tilde{\phi} \circ \tilde{\gamma}=\sigma$ follows from the metric completeness of $(X^{\ast},|\omega|)$ (see Lemma~\ref{lem:complete}).
\par
Analytic continuation along the path $\tilde{\gamma}$ proves that the function $g_{z} \circ \tilde{\phi}$ coincides with $f \circ \pi$. (Recall that $\pi$ is the projection of $\tilde{X^{\ast}}$ to $X^{\ast}$).)
\par
We first consider the case when $\pi \circ \tilde{\gamma}(1) \notin \mathcal{PPL}(\omega,f)$. Following Definitions~\ref{defn:localfactor} and~\ref{defn:PPL}, there is a neighbourhood of $\pi \circ \tilde{\gamma}(1)$ in $X^{\ast}$ in which  $f$ is factorised by $\phi$. Thus  $g_{z}$ extends to $\sigma(1)$ as a holomorphic function.
\par
If $\pi \circ \tilde{\gamma}(1) \in \mathcal{PPL}(\omega,f)$ but is not a zero of $\omega$, then $\phi$ is locally injective in a neighbourhood of $\pi \circ \tilde{\gamma}(1)$ and the meromorphic function $f$ can be factorised by $\phi$. Therefore $g_{z}$ extends to $\sigma(1)$ as a pole of a meromorphic function.
\par
In the latter case, $\pi \circ \tilde{\gamma}(1)$ is a zero  of $\omega$ of order $k$. Equivalently, $\tilde{\gamma}(1)$ is a critical point  of $\tilde{\phi}$ of order $k$. The equation $f \circ \pi = g \circ \tilde{\phi}$ can be locally solved at $\tilde{\gamma}(1)$  if we consider $g$ as a convergent Puiseux series with exponents belonging to  $\frac{1}{k+1}\mathbb{Z}$. This provides the correct extension of $g_{z}$ to $\sigma(1)$. The function $g_{z}$ cannot extend to $\sigma(1)$  holomorphically because in this case, $f$ will be locally factorised (see Definition~\ref{defn:localfactor}) by a primitive of $\omega$ and $\tilde{\gamma(1)}$ would not belong to the principal polar locus.
\end{proof}

\begin{corollary}\label{cor:CriticalNontrivial}
Consider a compact connected Riemann surface $X$  with a non-zero meromorphic $1$-form $\omega$ and a meromorphic function $f$ on $X$ such that $\mathcal{PPL}(\omega,f) \neq \emptyset$. Then any point $z \in X^{\ast} \setminus \mathcal{PPL}(\omega,f)$ has a finite critical radius $\rho(z)$.
\end{corollary}

\begin{proof}
We assume by contradiction that $\rho(z)=+\infty$, i.e. that $g_{z}$ is an entire function. Since $X$ is connected and $\mathcal{PPL}(\omega,f) \neq \emptyset$, there is a real-analytic path $\gamma:[0,1] \rightarrow X^{\ast}$ such that $\gamma(0)=z$ and $\gamma(1)=z_{0}$ where $z_{0} \in \mathcal{PPL}(\omega,f)$. The path $\gamma$ lifts to a path $\tilde{\gamma}$ on the universal cover $\tilde{X^{\ast}}$. The path $\tilde{\gamma}$ satisfies the assumptions of Lemma~\ref{lem:path} and we conclude that $g_{z}$ is not holomorphic at $\tilde{\phi} \circ \tilde{\gamma}(1)$ which is a contradiction.
\end{proof}

\subsection{Voronoi diagrams}\label{sub:Voronoi}
Let us now stratify the translation surface $(X,\omega)$ into cells according to the values of the \textcolor{blue}{\textit{Voronoi index}} defined below.

\begin{definition}\label{defn:VoronoiIndex}
Let $X$ be a compact connected Riemann surface $X$  with a non-zero meromorphic $1$-form $\omega$ and a meromorphic function $f$ such that $\mathcal{PPL}(\omega,f) \neq \emptyset$. For any point $z \in X^{\ast} \setminus \mathcal{PPL}(\omega,f)$, let $g_{z}$ be the Voronoi function of $z$ and  let $\rho(z)$ be the critical radius at $z$. 
\par
Then the \textcolor{blue}{\textit{Voronoi index}} $\nu_{z}$ is the number of points on the  circle of radius $\rho(z)$ centered at $0$ at which $g_{z}$ does not extend as a holomorphic function.
\end{definition}

\begin{lemma}\label{lem:Index}
For any point $z \in X^{\ast} \setminus \mathcal{PPL}(\omega,f)$, $\nu_{z}$ is a  positive integer.
\end{lemma}

\begin{proof}
Following Corollary~\ref{cor:CriticalNontrivial}, the radius $\rho(z)$ of convergence of $g_{z}$ at $0$ is finite and $g_{z}$ fails to extend to a holomorphic function  at least at one point of the circle of radius $\rho(z)$. At each of these singular points, $g_{z}$ extends as a Puiseux series converging in some neighbourhood of such a point (see Lemma~\ref{lem:path}). Compactness argument then proves that there can be at most finitely many such points inside the considered circle.
\end{proof}

All points in a punctured neighbourhood of a point $z\in \mathcal{PPL}(\omega,f)$ will satisfy $\nu=1$, and we extend the definition by setting $\nu_z=1$. Assuming that the principal polar locus is nonempty, the values of the Voronoi index decompose the underlying surface $X$ into:
\begin{itemize}
    \item the Voronoi cells (for which $\nu =1$);
    \item the Voronoi edges (for which $\nu =2$);
    \item the Voronoi vertices (for which $\nu \geq 3$).
\end{itemize}

\begin{definition}\label{defn:Voronoi}
Consider a compact Riemann surface $X$ with a non-zero meromorphic $1$-form $\omega$ and a meromorphic function $f$ on $X$ such that $\mathcal{PPL}(\omega,f) \neq \emptyset$. The \textcolor{blue}{\textit{Voronoi diagram}}  $\mathcal{V}_{\omega,f}$ of pair $(\omega,f)$ is the union of all points $z\in X$ such that $\nu_{z} \geq 2$; i.e. the union of the Voronoi edges and the Voronoi vertices.
\end{definition}

\begin{proposition}\label{prop:Voronoi}
Consider a compact Riemann surface $X$  with a non-zero meromorphic $1$-form $\omega$ and a meromorphic function $f$ on $X$ such that $\mathcal{PPL}(\omega,f) \neq \emptyset$. Then $\mathcal{V}_{\omega,f}$ is the union of geodesic segments in the (singular) flat metric $|\omega|$.
\end{proposition}

\begin{proof}
For any point $z_{0} \in \mathcal{V}_{\omega,f}$ of critical radius $r$ and Voronoi index $\nu$, let us  denote by $\alpha_{1},\dots,\alpha_{\nu}$ the points of $\lbrace{ t \in \mathbb{C}~\vert~|t|=r \rbrace}$ where the Voronoi function $g_{z_{0}}$ does not extend as a holomorphic function.
\par
Actually, there exists $\epsilon>0$ such that $g_{z_{0}}$ extends as a holomorphic function to a domain $\mathcal{D}$ formed by the points of the  open disk $\lbrace{ x \in \mathbb{C}~\vert~|x|<r+\epsilon \rbrace}$ which do not belong to the cuts $[1,\frac{r+\epsilon}{r}[\alpha_{i}$ for $1\leq i \leq \nu$.
\par
For any two distinct $i,j \in \lbrace{1,\dots,\nu \rbrace}$, let the line $L_{ij}$ be the midline of the segment $[\alpha_{i},\alpha_{j}]$, i.e. the line perpendicular to $[\alpha_i, \alpha_j]$ and passing through its midpoint. Denote by $\mathcal{L}$ the union of all these lines $L_{ij}$ for $1\le i<j\le \nu$.
\par
Consider a small neighbourhood $U$ of $z_{0}$ in $X$ in which $\omega$ has a well-defined primitive $\phi$ satisfying the condition $\phi(z_0)=0$. We can assume that for any $z \in U$, $|\phi(z)|<\delta$ where $\delta<\frac{\epsilon}{2}$. It follows that for any $z \in U$, the disk of convergence of $g_{z}$ in $\phi(z)$ is contained in $\mathcal{D}$, see Figure~\ref{fig:voronoi-straight}. Moreover, unless $\phi(z) \in \mathcal{L}$, we have $\nu(z)=1$.
\par
In the open set $U$, $\phi^{-1}(\mathcal{L})$ is the union of straight segments in the flat coordinate $\phi$ which finishes the proof.
\end{proof}

\begin{figure}[h]
    \centering
    \includegraphics[width = 0.9\linewidth]{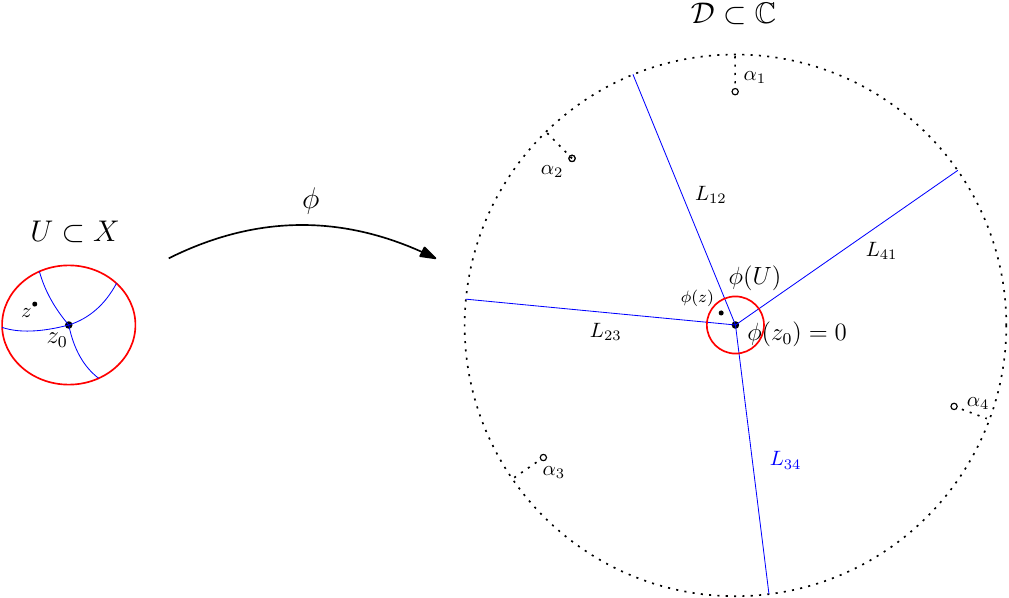}
    \caption{The Voronoi function $g_{z_0}$ extends to a holomorphic function in the domain $\mathcal{D}$. The blue lines form the union $\mathcal{L}$ of the midline rays contained in $\mathcal{D}$. The preimage $\phi^{-1}(\mathcal{L})$ is the union of the straight segments with respect to the coordinate $\phi$ in $U$.}
    \label{fig:voronoi-straight}
\end{figure}

\begin{remark} Observe that the 
Voronoi diagrams cannot be defined in  purely metric terms, as demonstrated by the following example. Consider $\omega=z dz$ and $f(z)=\frac{1}{z^{2}-1}$. Since $f$ is factorised by a primitive $\phi: z \mapsto \frac{1}{2}z^{2}$ of $\omega$, we have $\mathcal{PPL}(\omega,f) = \lbrace{ \pm 1 \rbrace}$.
\par
In the Euclidean metric, the points equidistant to $1$ and $-1$ satisfy the equation $\Re(z)=0$. However, for the Voronoi diagram associated with $\mathcal{PPL}(\omega,f)$, the points $1$ and $-1$ have the same image under the primitive $\phi$, and so there is no point $z$ in $\mathbb{C} \setminus \lbrace{ \pm 1 \rbrace}$ such that the corresponding Voronoi function $g_{z}$ has more than one singular point on its critical circle. Therefore the Voronoi diagram is empty and, in particular, it does not coincide with the imaginary axis as a purely metric definition might suggest.
\end{remark}

\subsection{Cauchy measure of a Voronoi diagram}\label{sub:Cauchy}

We shall define the \textcolor{blue} {\textit{Cauchy measure}} $\mu_{\omega,f}$ of the Voronoi diagram $\mathcal{V}_{\omega,f}$ by using central angles at points of $\mathcal{PPL}(\omega,f)$. It is obtained as the pullback of a certain differential form by a local primitive of $\omega$.

\begin{definition}\label{defn:angular}
Let $A,B$ be two distinct points in the complex plane. Denote by $L_{AB}$ their \textcolor{blue}{\textit {midline}}, i.e. the line perpendicular to the segment $[A,B]$ and passing through its middle point. We define the \textcolor{blue}{\textit{unnormalised Voronoi angular measure}} $\theta$ on  $L_{AB}$ as follows. 
For any segment $[M,N]$ contained in $L_{AB}$, let $\theta([M,N])$ be the angle at the vertex $A$ of the triangle $AMN$ (see Figure~\ref{fig:angular-measure}). Obviously,  $\theta(L_{AB})=\pi$.
\par
We define the \textcolor{blue}{\textit{normalised Voronoi measure}}  as $\Theta=\frac{1}{\pi}\theta$. It is  a probability measure on $L_{AB}$.  We note for the later use that the density of $\Theta$ on $L_{AB}$ is given by 
$$d\Theta_{AB}(P):=\frac{\vert A-B\vert }{2\pi\vert P-A\vert\vert P-B\vert}d\lambda(P),
$$  where $d\lambda$ is the Euclidean length  on $L_{AB}$.
\end{definition}

\begin{figure}[h]
    \centering
    \includegraphics[scale=0.9]{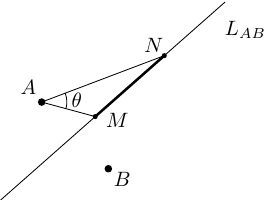}
    \caption{Illustration of the angular measure $\theta$ on the midline $L_{AB}$ of two points $A$ and $B$ in the plane. 
    }
    \label{fig:angular-measure}
\end{figure}

\begin{definition}\label{defn:Cauchy}
Consider a compact Riemann surface $X$ together  with a non-zero meromorphic $1$-form $\omega$ and a meromorphic function $f$ such that $\mathcal{PPL}(\omega,f) \neq \emptyset$.
\par
Then the \textcolor{blue}{\textit{Cauchy measure}} $\mu_{\omega,f}$ of the Voronoi diagram $\mathcal{V}_{\omega,f}$ is obtained by the integration of the differential form  obtained as  the pullback of the differential form $d\Theta_{AB}$  at each point $z$ such that  $\nu_{z}=2$,   by a local primitive $\phi$ of $\omega$ satisfying $\phi(z)=0$, see Definition~\ref{defn:angular}.
\end{definition}

\begin{remark}
Since  the principal polar locus contains  finitely many points  with a finite conical angle at each of them, the Cauchy measure $\mu_{\omega,f}$ of the whole Voronoi diagram is finite. Besides, it is invariant under the  rescaling $\omega \mapsto \lambda \omega$ where $\lambda \in \mathbb{C}^{\ast}$. 
\end{remark}


\section{Zero-free regions}\label{sec:ZeroFree}

In this section, we prove that zeros of iterates of the operator $T_{\omega}$ applied to a meromorphic function $f$ cannot accumulate inside Voronoi cells, i.e. in the complement  to  the Voronoi diagram $\mathcal{V}_{\omega,f}$. 

\begin{proposition}\label{prop:ZeroFree}
Let $X$ be a compact connected Riemann surface $X$  with a non-zero meromorphic $1$-form $\omega$ and a meromorphic function $f$ such that $\mathcal{PPL}(\omega,f) \neq \emptyset$.
\par
Then
$$\lim_{n\to+\infty}\dfrac{1}{n}\log \left| \frac{T_{\omega}^{n}f(z)}{n!}\right| = \log\frac{1}{\rho(z)}$$ 
on every compact subset of the union of the Voronoi cells, i.e., on the open subset $U \subset  X^{\ast\ast} 
$ formed by all points $z$ satisfying $\nu_{z}=1$, (see Definition~\ref{defn:VoronoiIndex}). Here $\rho(z)$ is the critical radius of the Voronoi function $g_{z}$, see Definition~\ref{defn:VorFunc}.
\par
In particular, for any compact subset $K\subset U$, there is a positive integer  $N$ such that for any $n>N$, $T_{\omega}^{n}(f)$ has no zeros in $K$.
\end{proposition}

Since the distance from any compact set in $U$ to a singularity, and to the boundary, is bounded below by a positive constant, it suffices to prove a local result. Proposition~\ref{prop:ZeroFree} will be deduced from a purely analytic lemma about the asymptotics of the coefficients of the Taylor series of a Puiseux series, due to Orlov in \cite{Orlov}. 
Lemma~\ref{lem:coeffs-conv} applies a uniform version of his result, see our Corollary \ref{cor:Orlov}.

\begin{lemma}\label{lem:coeffs-conv}
Let $g$ be a holomorphic function on the open centered disk $D(0,R) \subset \mathbb{C}$ of radius $R>0$. Suppose that $g$ can be extended holomorphically to $\partial D(R)$, except at only one point $d$ on the boundary circle $\partial D(R)$ around which it extends as a Puiseux series (with non-integer exponents). Then, there exists $\delta>0$ such that 
$$\frac{1}{n}\log \left|\frac{g^{(n)}(x)} {n!}\right| \to \log \frac{1}{|x -d|}$$ uniformly on the closed disk $\overline{D}(0,\delta)$ of radius $\delta$.
\end{lemma}\label{lem:disk}

\begin{proof}
Suppose that the Puiseux series at $d$ has a leading exponent $m_0$, meaning that $B_{m_0}(d-x)^{m_0}$ is the term with the smallest non-integer or negative exponent.

By Corollary~\ref{cor:Orlov}, there exist $\epsilon, \delta > 0$ such that on $\overline{D}(0,\delta)$ one has the  asymptotic behavior  
$$\frac{g^{(n)}(x)}{n!} = b(n,x)(d-x)^{-n} + O\left((\vert d-x\vert (1+\epsilon))^{-n}\right),$$
 where $$
b(n,x)=K (x-d)^{m_0}n^{-(m_0+1)}
(M+
O(n^{-\frac{1}{M}}\vert x-d\vert^{\frac{1}{M}})),
$$
with all constants depending only on $g$ and $R$.
Then,
\begin{equation}\label{eq: before-log}
\left\vert\frac{g^{(n)}(x)}{n!}\right\vert = |d-x|^{-n} \left\vert b(n,x) + O((1+\epsilon)^{-n})\right\vert.
\end{equation}
We claim that the expression $\frac{1}{n}\log|b(n,x)+O((1+\epsilon)^{-n})|$ converges  as $n \to +\infty$ to 0 uniformly on $\overline{D}(0,\delta)$. It suffices to check that $\log|b(n,x)|$ grows at a slower rate than a power of $n$ in $\overline{D}(0,\delta)$. This holds because there exist constants $C_1, C_2$ for which the inequality 
$$C_1n^{-(m_0+1)} \leq |b(n,x)| \leq C_2n^{-(m_0+1)}$$
holds for all $x$ and sufficiently large $n$. 
Therefore, taking log of \eqref{eq: before-log}, dividing by $n$, and taking limits results in
$$\lim\limits_{n \to +\infty}\frac{1}{n}
\log\left|\frac{g^{(n)}(x)}{n!}\right| = -\log|x-d| $$
where the convergence is uniform in $\overline{D}(\delta)$ as claimed.
\end{proof}

We now prove Proposition~\ref{prop:ZeroFree} by applying Lemma~\ref{lem:coeffs-conv} to the Voronoi function of a point contained in a Voronoi cell.

\begin{proof}[Proof of Proposition~\ref{prop:ZeroFree}]
For any $z_{0} \in U$, the Voronoi function $g_{z_{0}}$ converges in the open disk of radius $\rho(z_{0})$ centered at $0$ and extends as a holomorphic function to every point of its boundary circle, except at one point $d$, to a neighbourhood of which $g_{z_0}$ extends as a Puiseux series (see Lemma~\ref{lem:path} and Definition~\ref{defn:VoronoiIndex}). On any connected compact subset $K\subset U$ the distance to the singular point $d$ is bounded below by a non-zero constant. Using Lemma~\ref{lem:coeffs-conv} we then deduce  the existence of a neighbourhood $W$ of $0$ in which  the sequence of functions ${\frac{1}{n}}\log\left\vert\frac{g^{(n)}_{z_{0}}(x)}{n!}\right\vert$ converges to $-\log |x -d|$ uniformly. Therefore, there is a neighbourhood of $z_{0}$ in $X$ in which  $${\frac{1}{n}}\log\left\vert\frac{T_{\omega}^{n}(f)(z)}{n!}\right\vert = {\frac{1}{n}}\log\left\vert \frac{g^{(n)}_{z_{0}}(\phi(z))}{n!}\right\vert $$  converges uniformly to $- \log {\rho(z)}$ where $\rho(z)$ is the radius of convergence of $g_{z}$. The last claim follows immediately.
\end{proof}

\section{Edges in the limit set}\label{sec:Edges}

The remaining part of the proof of Theorem~\ref{thm:MAIN} is an application of certain results of the logarithmic potential theory. We do this in three steps. The first will show that  
$$\lim_{n\to +\infty} \dfrac{1}{n}\log \left| \frac{T_{\omega}^{n}f(z)}{n!}\right| = \log\frac{1}{\rho(z)}$$ in the space $L^{1}_{loc} (X^{\ast})$ of locally integrable functions on $X^{\ast}$ (with respect to the measure $\omega \wedge \overline{\omega}$). In addition to the uniform convergence version of Orlov's theorem, we need to analyse what happens on the Voronoi diagram, using the minimal modulus theorem. The next step is to note that $L^{1}_{loc} (X^{\ast})$-convergence automatically implies convergence of Laplacians, for the Laplace operator $\Delta_{\omega}$ induced by the flat metric $|\omega|$ on $X^{\ast}$. The third step is identifying these Laplacians. We first note that, up to normalisation, the Laplacian of $\dfrac{1}{n}\log \left| \frac{T_{\omega}^{n}f(z)}{n!}\right|$ considered as a distribution coincides with the root-counting measure of $T_{\omega}^{n}f(z)$. As a distribution the sequence of these Laplacians converges to the Laplacian of $\log\frac{1}{\rho(z)}$, which we then calculate to be $\frac{\mu_{\omega,f}}{A}+\frac{1}{A} \sum\limits_{p \in \mathcal{P}} (d_{p}-1)\delta_{p}$. This claim  is exactly the second part of Theorem~\ref{thm:MAIN}. After that we will deduce the first part of Theorem~\ref{thm:MAIN} from this convergence. 

\subsection{Convergence in $L^{1}_{loc}(X^{\ast})$}\label{sub:L1LocConvergence}

Proposition~\ref{prop:ZeroFree} shows that $\frac{1}{n}\log \left| \frac{T_{\omega}^{n}f(z)}{n!}\right|$ converges uniformly to $\log\frac{1}{\rho(z)}$ on every compact subset inside the union of the Voronoi cells of $X$, not containing the singularities. Oscillations of $T_{\omega}^{n}f(z)$ near the edges of the Voronoi diagram force  us to work with the weaker type of convergence, namely, with the 
$L^{1}_{loc}$-convergence.

\begin{proposition}\label{prop:L1LocConvergence}Let $\omega$ be a meromorphic 1-form and $f$ a meromorphic function on a compact Riemann surface $X$. Assume that $\mathcal{PPL}(\omega,f)\neq \emptyset$. For any $z \in X^{\ast}$, there exists a neighbourhood $U$ of $z$ such that on any compact subset $K \subset U$, one has
$$
\lim_{n\to +\infty}\int\limits_{K} \bigg\vert \frac{1}{n}\log \big\vert  \frac{T_{\omega}^{n}f(z)}{n!}\big\vert - \log\frac{1}{\rho(z)} \bigg\vert \ \omega \wedge \overline{\omega} = 0,
$$
where $\rho(z)$ is the critical radius of the Voronoi function $g_{z}$.
\end{proposition}

To do this, we first need to check that $\log\frac{1}{\rho(z)}$ is a locally integrable function.

\begin{lemma}\label{lem:RadiusLocallyIntegrable}
In the notation as Proposition \ref{prop:L1LocConvergence}, assuming that $\mathcal{PPL}(\omega,f)\neq \emptyset$, the function $z \mapsto \log\frac{1}{\rho(z)}$ belongs to $L^{1}_{loc} (X^{\ast})$.
\end{lemma}

\begin{proof}
For any $z_{0} \in X^{\ast} \setminus \mathcal{PPL}(\omega,f)$,  consider its Voronoi function $g_{z_{0}}$ and a local primitive $\phi$ of $\omega$ such that $\phi(z_{0})=0$. Let $d_{1},\dots,d_{k}$ be the points of the critical circle where $g_{z_{0}}$ does not extend holomorphically, with $k$ being the Voronoi index of $z_{0}$, see Lemma~\ref{lem:Index}. Then, there is a neighbourhood $U$ of $z_{0}$ such that for any $z \in U$, $\rho(z)= \min\limits_{1 \leq i \leq k} |\phi(z)-d_{i}|$.
\par
Calculating a simple integral we can  show that the function $w \mapsto \log\vert w-b\vert$ for some constant $b \in \mathbb{C}$ is locally integrable, in any compact set, possibly containing $b$.
It follows that $x \mapsto - \min\limits_{1 \leq i \leq k} \log|x-d_{i}|$ is locally integrable in $\mathbb{C}$.
We deduce that $z \mapsto \log\frac{1}{\rho(z)}$ is integrable in a small enough neighbourhood of $z_{0}$ in $ X^{\ast}$. The claim follows.
\end{proof}

Proposition~\ref{prop:L1LocConvergence} follows from Lemma \ref{lem:LocConvergencePuiseux} below whose proof is postponed to Section~\ref{sub:minmodulus}.

\begin{lemma}\label{lem:LocConvergencePuiseux}
Let $g$ be a holomorphic function on the open centered disk (at 0) $D(\rho) \subset \mathbb{C}$ of radius $\rho>0$. Suppose that $g$ can be extended holomorphically to $\partial D(\rho)$, except for finitely many points $d_{1},\dots,d_{k}$ on the boundary circle $\partial D(\rho)$ where it extends as a Puiseux series. Then, setting $$M(x)= - \min\limits_{1 \leq i \leq k} \log|x-d_{i}|,$$ there exists $\delta>0$ such that 
$$
\int\limits_{\overline{D}(\delta)} \left\vert \frac{1}{n}\log \big|\frac{g^{(n)}(x)} {n!}\big| - M(x) \right\vert d\lambda \longrightarrow 0
$$
where $\overline{D}(\delta)$ is the closed centered disk of radius $\delta$ and $d\lambda$ is the standard Lebesgue measure.
\end{lemma}

\begin{proof}[Proof of Proposition~\ref{prop:L1LocConvergence}]
For any $z_{0} \in X^{\ast} \setminus \mathcal{PPL}(\omega,f)$, we can apply Lemma~\ref{lem:LocConvergencePuiseux} to the Voronoi function $g_{z_{0}}$. Notice that $\omega \wedge \overline{\omega} = \phi^{\ast} d\lambda$ and $T_{\omega}^{k} (f) = g^{(k)} \circ \phi$. This shows the required  convergence to 0 in a compact set not containing $\mathcal{PPL}(\omega,f)$, and it remains to prove that in a small enough neighbourhood of a point $a$ in $\mathcal{PPL}(\omega,f)$, the integral can be made arbitrarily small independently of $n$. We argue as follows. Note that if the statement is true for the derivative $g^{(k)}$ then it is also true for $g$, since 
$\frac{1}{n} \log\big( \frac{(k+n)!}{n!} \big)\to 0$ as $n\to +\infty$. Hence we take a derivative of our function of sufficiently large order and, in particular, assume that $m_0+1\leq - \frac{1}{M}$ in a neighbourhood given by $\vert z-a\vert <1$. This allows us to rewrite the asymptotic estimate of Corollary~\ref{cor:Orlov} as
\begin{equation}\label{eq: last-log}
\left\vert\frac{g^{(n)}(x)}{n!}\right\vert = |a-x|^{-n+m_0}n^{-(m_0+1)}(C + B(z,k)),
\end{equation}
where $C\neq 0$ and $B(z,k)$ goes uniformly in $z$ to 0 as $k\to +\infty$ (see Remark~\ref{rmk:orlov}).
Taking the logarithm of the absolute value, dividing by $n$, and applying the triangle inequality, we see that the integral over a small neighbourhood of $a$ behaves like a multiple of the integral of $\log\vert x-a\vert$, i.e. it can be made arbitrarily small by choosing a sufficiently small neighbourhood. This proves the proposition.
\end{proof}

\subsection{Minimum modulus principle} 
\label{sub:minmodulus}

This subsection is devoted to the proof of  Lemma~\ref{lem:LocConvergencePuiseux}. We recall that $d_{1},\dots,d_{k}$ are the points on the critical circle to which $g$ does not extend holomorphically. Let us  assume that $k \geq 2$ since  if $k=1$, Lemma~\ref{lem:LocConvergencePuiseux} immediately follows from Lemma~\ref{lem:coeffs-conv}.
\par
Assuming that $d_{1},\dots,d_{k}$ are cyclically ordered by their arguments, we introduce $k$ points $c_{1},\dots,c_{k}$ on the critical circle of radius $\rho$ in such a way that $c_{i}$ is the midpoint of the circle arc $(d_{i}, d_{i+1})$. We define the subset $A \subset \mathbb{C}$ formed by the  $k$ segments $[0,c_{1}],\dots,[0,c_{k}]$. We denote by $A_{\epsilon}$ the open tubular $\epsilon$-neighbourhood of $A$ in the centered disk $D(\rho)$ (centered at the origin), see Figure~\ref{fig:tubular Voronoi}.
\par
\begin{figure}[h]
    \centering
\includegraphics[width=0.35\linewidth]{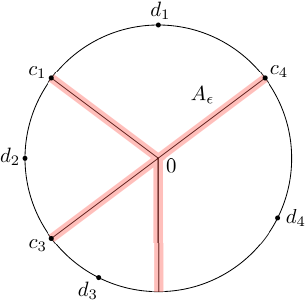}
    \caption{The open tubular neighbourhood $A_{\epsilon}$ of $A$.} 
    \label{fig:tubular Voronoi}
\end{figure}

In order to prove Lemma~\ref{lem:LocConvergencePuiseux}, we need, in particular, to show that the contribution of an $\epsilon$-neighbourhood of $A$ to the value of the integral of the function  $\left\vert \frac{1}{n}\log \big|\frac{g^{(n)}(x)} {n!}\big\vert \right\vert$ can be made arbitrarily small by the choice of $\epsilon$. The rest of the integral will be computed/estimated  using Lemma~\ref{lem:coeffs-conv}.

\begin{lemma}\label{lem:MainEstimate}
Under the assumptions of Lemma~\ref{lem:LocConvergencePuiseux}, there exists a constant $C$ depending only on $g$ such that for any small enough $\epsilon>0$ and $\delta>0$, and all large enough $n$, one has the inequality
$$
\frac{1}{n}\int\limits_{\overline{A_{\epsilon}\cap D(\delta)}} \left\vert \log \big|\frac{g^{(n)}(x)} {n!}\big\vert \right\vert d\lambda
< C \epsilon, 
$$
where $d\lambda$ is the standard Lebesgue measure on $\mathbb{C}$ and $\overline{A_{\epsilon}\cap D(\delta)}$ is the closure of the intersection of $A_{\epsilon}$ with the disk of radius $\delta$ centered at the origin
\end{lemma}

The key ingredient of the proof is the classical Cartan-Boutroux lemma (i.e.,  the Minimum Modulus Principle) which provides an estimate of the growth of $\log \vert g(x)\vert$ near its singularities in a form  suitable for integration. For the sake of completeness, let us  quote the following result  from \cite{L} (see also \cite{AZ}). (In Lemma~\ref{lem:MinimumPrinciple}, $e$ stands for the standard Euler  number, i.e. the base of the natural logarithm.)

\begin{lemma}(\cite[Lecture 11, Theorem 4]{L})\label{lem:MinimumPrinciple}
Consider a holomorphic function $g$ on the closed disk $\overline{D(0,2er)}$ with $\vert g(0)\vert=1$ and $r>0$. Let $\eta$ be an arbitrary number satisfying $0 < \eta < 3e/2$ and  denote by $M$ the maximal value of $\big \vert g(x)\big \vert$ on the circle $\vert x\vert=2er$.
\par
Then there exists a family of (excluded) 
disks in $\overline{D(0,2er)}$ the sum of radii of which does not exceed $4\eta r$ and such that in the complement to the union of these disks one has the inequality
 \begin{equation}
 \label{eq:CB}
 \log\vert g(x)\vert >-H(\eta)\log M,
 \end{equation} 
 where $H(\eta)=2+\log\frac{3e}{2\eta}$.
 \end{lemma}

\begin{remark}\label{rem:H}
We want to apply this result for small $\eta$. Note that our  condition on $\eta$ implies that:
\begin{enumerate}
    \item $H(\eta)\geq 2>0$;
    \item $H(\eta)\sim - \log \eta$ as $\eta\to 0$; 
    \item the total area of the excluded disks in  Lemma~\ref{lem:MinimumPrinciple} is less than $\pi(4\eta R)^2$ (since the maximal area occurs in the case of one excluded disk).
\end{enumerate}
\end{remark}

\begin{remark}\label{rem:normalisation}
If $g(0) \neq 0$, we apply  Lemma~\ref{lem:MinimumPrinciple} to $g(z)/g(0)$ and modify \eqref{eq:CB} obtaining
 \begin{align} 
 \log\vert g(x)\vert-\log\vert g(0)\vert &> -H(\eta)(\log M-\log\vert g(0)\vert) \notag   \\
 \iff \log\vert g(x)\vert &>-H(\eta)\log  M+(1+H(\eta))\log\vert g(0)\vert. \label{eq:CB2}
 \end{align} 
\end{remark}

We will also make use of a specific covering of $A_{\epsilon}$ by a family of disks. 

\begin{lemma}\label{lem:VoronoiDiskCovers}
For every sufficiently small $\epsilon>0$, there exist $r>0$, $h>0$, and $s_{\epsilon}$ points $x_{1},\dots,x_{s_{\epsilon}} \in A_{\epsilon} \setminus A$ such that the following $3$ conditions are satisfied. 
\begin{enumerate}
    \item $A_{\epsilon}$ is covered by the union of the disks $\bigcup\limits_{i=1}^{s_{\epsilon}}D(x_i,r)$ of radius $r$ centered at $x_i$.
    \item The area of the union $\bigcup\limits_{i=1}^{s_{\epsilon}}D(x_i,r)$ is smaller than $ \epsilon C_{A}$ where $C_{A}$ is a constant depending only on $A$.
    \item The union $\bigcup\limits_{i=1}^{s_{\epsilon}}D(x_i,2er+h)$ lies at the distance at least $d/2$ from the points $\{d_1,\ldots, d_k\}$ where $d$ is the distance between $A$ and $\{d_1,\ldots, d_k\}$.
\end{enumerate}
\end{lemma}

\begin{proof}
Observe that for (3) to be true it suffices that $r$ and $\epsilon$ satisfy the inequality $2er+h<d/2-\epsilon$. Given  $\epsilon<d/2$, this is clearly a feasible constraint on $r$ and $h$. The rest of the argument is intuitively clear and can be made precise by using  Euclidean geometry. Namely, the domain $A_{\epsilon}$ decomposes into a central $k$-gon (contained in a disk of radius $\epsilon$) and $k$ strips (each contained in a rectangle of length $\rho$ and width $2\epsilon$), see Figure~\ref{fig:tubular Voronoi}. Each of these polygonal pieces can be covered by a family of equilateral triangles of size $\epsilon$ where the  number of these equilateral triangles grows at most  as $\frac{\rho}{\epsilon}$. We construct a family of disks, each circumscribing an equilateral triangle of the family. Then, the linear  in $\epsilon$ bound on the area of the covering family of disks is easily realised. An arbitrarily small deformation of the covering ensures that the centers of the disks do not belong to $A$. Finally, we notice that since $r = r(\epsilon) \to 0$ as $\epsilon$, we have $\lim\limits_{\epsilon \to 0} d/2 - \epsilon - 2er = d/2$, and we can choose the constant $h$ such that it works for all sufficiently small  $\epsilon$. 
\end{proof}

Let us now settle Lemma~\ref{lem:MainEstimate}.

\begin{proof}[Proof of Lemma~\ref{lem:MainEstimate}]
The argument consists of four steps.

\subsubsection{First step}\label{subsub:First}
Choose $0< \epsilon < \delta < \rho$ so  small  that $g$ is holomorphic in $A_{\epsilon} \cap D(\delta)$. Let $r>0$ be such that  $\overline{A_{\epsilon}\cap D(\delta)}$ is covered by the disks $D(x_i,r)$, $i=1, \ldots, s_{\epsilon}$ as in Lemma~\ref{lem:VoronoiDiskCovers}. Decreasing if necessary $\delta$ and $\epsilon$,  we require that each $D(x_i,2er+h)$ is contained in $D(\rho)$, where $h$ is a constant independent of $\epsilon$, see the end of the proof of Lemma~\ref{lem:VoronoiDiskCovers}. Then, since $x_{i}$ lies outside of $A$, Lemma~\ref{lem:coeffs-conv} implies that
\begin{equation}
\label{eq: 2nd}
\lim_{n\to +\infty} \frac{1}{n}\log\bigg\vert
\frac{ g^{(n)}(x_i)}{n!}\bigg\vert =-\log\vert x_i-d_i\vert=M(x_i),
\end{equation}
where $d_i$ is the point nearest to $x_i$ among $\{ d_1,\ldots, d_k\}$. In particular, 
$\frac{ g^{(n)}(x_i)}{n!}\neq 0$ for all sufficiently large $n$. Let $n_0 \in \mathbb{N}$ be such that for all $i=1,\dots,s_\epsilon$, 
\begin{equation}
\bigg\vert M(x_i)- \frac{1}{n}\log\big\vert
\frac{ g^{(n)}(x_i)}{n!}\big\vert \bigg\vert \leq \vert M(x_i)\vert \text{ if } n>n_0.
\end{equation}
Therefore,
\begin{equation}
\label{eq: 2nd:2}
\bigg\vert \log \big\vert
\frac{ g^{(n)}(x_i)}{n!}\big\vert \bigg\vert\leq C \cdot n, \text{ if } n>n_0,
\end{equation}
where $C$ is the maximum of $2\vert M(x)\vert$ in a region containing $\overline{A_{\epsilon}\cap D(\delta)}$.

\subsubsection{Second step}\label{subsub:Second}
Condition (3) in Lemma~\ref{lem:VoronoiDiskCovers} and the choice of a small $\delta>0$ ensure that $g$ is holomorphic on the union of the disks $\bigcup\limits_{i=1}^{s_\epsilon}D(x_i,2er+h)$. Let $B$ be the maximal value of $\vert g(x)\vert$ on this union. Since circles of radius $h$ around a point $x$ on the circle $\vert x-x_i\vert = 2er$ are contained in $D(x_i,2er+h)$, we can apply Cauchy's bound for the (absolute) values of the derivatives of $g$ at a point $x$ with $\vert x - x_i \vert \leq 2er$ and obtain an estimate 
\begin{equation}
\label{eq:Cauchyderivative}
 \bigg \vert \frac{g^{(n)}(x)}{n!}\bigg\vert \leq h^{-n} B.
\end{equation} 
This inequality is hence true for any $x\in \bigcup\limits_{i=1}^{s_\epsilon} D(x_i,2er)$.

\subsubsection{Third step}\label{subsub:Third}

We now apply Lemma~\ref{lem:MinimumPrinciple} to $\frac{g^{(n)}(x)}{n!}$ in each disk $D_i = D(x_i,r)$. Let $M_n \leq h^{-n}B$ be the maximum  of $\big\vert \frac{g^{(n)}(x)}{n!}\big\vert$ on the circle $\vert x - x_i\vert = 2er$. This inequality follows from \eqref{eq:Cauchyderivative}.  

\medskip
For small $\eta > 0$, let us  introduce the subset $O_{\eta,i}\subset D_i$ such that for $x\in O_{\eta,i}$,  we have
\begin{equation}
\label{eq:CB3}
\log\left\vert \frac{g^{(n)}(x)}{n!}\right\vert>-H(\eta)\left(\log  M_{n}-\log\left\vert \frac{g^{(n)}(x_i)}{n!}\right\vert\right)+\log\left\vert \frac{g^{(n)}(x_i)}{n!}\right\vert.
 \end{equation}
Following Lemma~\ref{lem:MinimumPrinciple} and Remark~\ref{rem:normalisation}, the subset $O_{\eta,i}$ satisfies the relation 
 \begin{equation}\label{eq:areaBound}
\lambda(D_i\setminus O_{\eta,i}) \leq \pi(4\eta r)^2.
\end{equation}
Inequality \ref{eq:areaBound} gives  a rough  estimate of $\big\vert \log\vert \frac{g^{(n)}(x)}{n!}\vert\big\vert$ as follows. 
First, by \eqref{eq:Cauchyderivative},
 \begin{equation}
 \label{eq:absolutevalue1}
\log(h^{-n}B)\geq \log\big\vert \frac{g^{(n)}(x)}{n!}\big\vert.
 \end{equation}
 Then, by the Maximum Modulus principle, 
 $$\log(h^{-n}B)-\log\left\vert \frac{g^{(n)}(x_i)} {n!}\right\vert \geq \log  M_n-\log\left\vert \frac{g^{(n)}(x_i)}{n!}\right\vert\geq 0.$$ Using the latter inequality  and the fact that $-H(\eta)<0$, we have that \eqref{eq:CB3} implies 
 \begin{equation}
 \label{eq:absolutevalue2}
 \log\left\vert \frac{g^{(n)}(x)}{n!}\right\vert \geq -H(\eta)\left(\log(h^{-n}B)-\log\left\vert \frac{g^{(n)}(x_i)}{n!}\right\vert\right)+\log\left\vert \frac{g^{(n)}(x_i)}{n!}\right\vert. 
 \end{equation}
 The standard triangle inequality then implies that for any triple of real numbers $c\geq b\geq a$, we have $\vert a\vert+\vert c\vert \geq \vert b\vert.$ 
 So, \eqref{eq:absolutevalue1} and \eqref{eq:absolutevalue2}, together with \eqref{eq: 2nd:2}, imply
 \begin{equation}
  \label{eq:slut}
        (\vert H(\eta) \vert+1)(\vert\log ( h^{-n}B)\vert +        Cn) \geq  
        \bigg \vert \log\big\vert \frac{g^{(n)}(x)}{n!}\big\vert \bigg \vert.
  \end{equation}  
Since $\bigcup\limits_{i=1}^{s_{\epsilon}}D(x_i,r)= \bigcup\limits_{i=1}^{s_{\epsilon}} D_{i}$ covers $A_{\epsilon} \cap D(\delta)$, we can finally deduce from \eqref{eq:areaBound} that the subset $O_\eta=\bigcup\limits_{i=1}^{s_{\epsilon}} O_{\eta,i} \subset \bigcup\limits_{i=1}^{s_{\epsilon}}D_i$, for which the inequality \eqref{eq:slut} holds, satisfies the relation 
\begin{equation}
\label{eq:exhausting2}
\lambda (\overline{A_{\epsilon}\cap D(\delta)} \setminus O_\eta)\leq s_\epsilon \pi(4\eta r)^2=L\eta^2 \epsilon,
\end{equation} 
for some constant $L$.

\subsubsection{Fourth step}\label{subsub:Fourth}
Fix some $\epsilon>0$. The inequality \eqref{eq:CB3} implies that as $\eta \to 0$ the sets $O_{\eta}$ form an exhausting sequence of open subsets for  $\overline{A_{\epsilon} \cap D(\delta)}$. 

Now, we can estimate the integral  
$$
\frac{1}{n}\int\limits_{\overline{A_{\epsilon}\cap D(\delta)}} \left\vert \log \big|\frac{g^{(n)}(x)} {n!}\big\vert \right\vert d\lambda.
$$
Set $\eta_{j}=1/j$ and $O_{0}:=\emptyset$ and 
subdivide the integral as follows:
$$
\frac{1}{n}\int\limits_{\overline{A_{\epsilon}\cap D(\delta)}} \left\vert \log \big|\frac{g^{(n)}(x)} {n!}\big\vert \right\vert d\lambda
=
\frac{1}{n}
\sum\limits_{j=0}^{+\infty} \int_{O_{\eta_{j+1}}\setminus O_{\eta_{j}} } \left\vert \log \big|\frac{g^{(n)}(x)} {n!}\big\vert \right\vert d\lambda.
$$
By \eqref{eq:slut} and \eqref{eq:exhausting2}, the last sum is dominated by
 \begin{equation}
 \label{eq:finalsum}
\frac{1}{n}\sum\limits_{j=1}^{+\infty} \frac{L\epsilon}{j^2}\vert (\vert H(1/j) \vert+1)(\vert\log ( h^{-n}B)\vert   +     Cn).
 \end{equation}  
Now, observe that $\vert H(1/j)\vert+1\sim  \log j $ by Remark~\ref{rem:H}. Secondly,  $\frac{\vert\log ( h^{-n}B)\vert   +  Cn}{n}$ is globally bounded (recall that $h$ is a constant, independent of $\epsilon$). Hence, \eqref{eq:finalsum} converges to $\epsilon S_n$, for some uniformly bounded $S_n$. This is the desired bound, and Lemma~\ref{lem:MainEstimate} is proved.

\subsubsection{Proof of Lemma~\ref{lem:LocConvergencePuiseux}}

By choosing sufficiently small $\epsilon$ and $\delta$, Lemma~\ref{lem:coeffs-conv} guarantees that the integral of $\left\vert \frac{1}{n}\log \left|\frac{g^{(n)}(x)} {n!}\right| - M(x) \right\vert d\lambda$ on $\overline{D}(\delta) \setminus A_{\epsilon}$ converges uniformly to zero. It remains to show that the integral of the same measure over $\overline{A_{\epsilon}\cap D(\delta)}$ can be made arbitrarily small as $\epsilon \to 0$.
\par
By choosing a small enough $\delta$, we achieve that $|M(x)|$ can be globally bounded on
$\overline{D}(\delta)$ while  the area of $A_{\epsilon}\cap D(\delta)$ can be bounded by a linear expression in $\epsilon$. Together we obtain that 
$$
\int\limits_{\overline{A_{\epsilon}\cap D(\delta)}} \left\vert M(x) \right\vert d\lambda < B \cdot \epsilon
$$
for some constant $B>0$.
\par
Using the estimate of Lemma~\ref{lem:MainEstimate} (proved in Sections~\ref{subsub:First} -~\ref{subsub:Fourth}), we conclude  that 
$$
\int\limits_{\overline{A_{\epsilon} \cap D(\delta)}} \left\vert \frac{1}{n}\log \left|\frac{g^{(n)}(x)} {n!}\right|  \right\vert d\lambda <C\cdot \epsilon
$$
for some constant $C$ independent  of $n$. 
\par
By the triangle inequality, we deduce that $\int\limits_{\overline{A_{\epsilon} \cap D(\delta)}}  \left\vert \frac{1}{n}\log \left|\frac{g^{(n)}(x)} {n!}\right| - M(x) \right\vert d\lambda$ can be made arbitrarily small as $\epsilon \to 0$. This proves the convergence of the integral over $\overline{D}(\delta)$ to zero.
\end{proof}

\subsection{Laplace operator acting on $L^{1}_{loc} (X^{\ast})$}\label{sub:Laplacian}

As discussed in Section~\ref{sub:FlatMetric}, on the surface $X^{\ast}$ obtained from $X$ by puncturing it at the poles of $\omega$, $|\omega|$ is a singular flat metric with conical singularities at the zeros of $\omega$. Let us denote by $\Delta_{\omega}$ the corresponding Laplace operator defined on $L^{1}_{loc} (X^{\ast})$. We follow standard treatments of Laplace operators for singular flat metrics by Hillairet and Kokotov, see \cite{Hi, Ko} for details. With respect to the parametrising variable it coincides with the standard Laplace operator $\Delta=\frac{\partial^2}{\partial x^2}+ \frac{\partial^2}{\partial  
y^2}$ acting on test functions, that are supported even at the singular points of the flat metric, see \cite[Proposition 3.3 and (4.3)]{Hi}. Note that then as classically (see \ref{sec:puiseauxlaplacian})
$$
\Delta_\omega\log\vert g\vert =2\pi\sum_p\text{ord}_p\delta_p.
$$
This allows us to compute the Laplacian of $T_{\omega}^{n}f$, considered as a distribution. 
\begin{lemma}\label{lem:LaplaceF}
For some $N \in \mathbb{N}$ and  any $n \geq N$, one has $$\frac{1}{2\pi}\Delta_{\omega} \log \left\vert T_{\omega}^{n}f(z) \right\vert =
\sum_{z\in Z(T_{\omega}^{n}f)}\ \delta_z - n \left( \sum_{p\in  \mathcal{PPL}(\omega,f)} \delta_p \right) - \nu, 
$$
where $Z(T_{\omega}^{n} f)$ is the set of the zeros of $T_{\omega}^{n} f$ in $X^{\ast}$ (counted with multiplicities) and $\nu$ is a fixed measure of finite mass and support.
\end{lemma}

\begin{proof}
Take $z_{0}\in X^{\ast\ast}$ obtained from $X$ by puncturing it at the zeros and the poles of $\omega$. For any small enough neighbourhood $U$ of $z_{0}$, a local primitive $\phi$ of $\omega$ such that $\phi(z_{0})=0$ defines a local flat chart in which the operator $\Delta_{\omega}$ is conjugated to the standard Laplace operator $\Delta$. Since  $T_{\omega}^{k} (f) = g^{(k)} \circ \phi$ in $U$, we have
\begin{equation}
\label{eq:primitivef}
\Delta_{\omega} \log \left\vert T_{\omega}^{n}f(z) \right\vert
=
(\Delta \log \left\vert g^{(n)} \right\vert) \circ \phi(z)
.
\end{equation}
It is well-known that the standard Laplacian of the logarithm of the absolute value of a meromorphic function $f$, considered as a $L^1_{loc}$-function (and hence as a  distribution), is the divisor of $f$ expressed in terms of Dirac measures. We deduce that on any open subset of $X^{\ast\ast}$, $\frac{1}{2\pi}\Delta_{\omega} \log \left\vert T_{\omega}^{n}f(z) \right\vert$ coincides with the divisor of the zeros and poles of $T_{\omega}^{n}f$ (taken as a distribution). Following the results of Section~\ref{sub:Growth}, the poles of $T_{\omega}^{n}f$ lying in  $X^{\ast\ast}$ are located at the poles of $f$. They belong to $\mathcal{PPL}(\omega,f)$ and the order of each of them increases by one after each application of $T_{\omega}$. This proves the required claim for any open subset of $X^{\ast}$ not containing a zero of $\omega$.
\par
Now  consider an arbitrarily small open neighbourhood $U$ of a zero $z_{0}$ of $\omega$. There  locally exists a primitive $\zeta=\phi$ of $\omega$ that maps the open set $U$ to an open set $V$ in the $\zeta$-plane as a branched cover.
\par
In the first case, $z_{0} \notin \mathcal{PPL}(\omega,f)$ and $f=g\circ \phi$ is locally factorised by a primitive $\zeta=\phi$ of $\omega$ (see Definition~\ref{defn:localfactor}). By the remark on the Laplacian before the proof, the equality \eqref{eq:primitivef} still holds, and the same argument as above gives that 
$$
\frac{1}{2\pi}\Delta_{\omega} \log \left\vert T_{\omega}^{n}f(z) \right\vert=\sum_{z\in Z(T_{\omega}^{n}f)\cap U}\ \delta_z
$$
as a distribution on $X^{\ast}$ (where we use the local coordinates of the complex structure on the Riemann surface).
\par
In the second case, $z_{0} \in \mathcal{PPL}(\omega,f)$ and $T_{\omega}^{n}f$ has a pole in $z_{0}$ provided $n$ is large enough, see Proposition~\ref{prop:growth}. Moreover, if $U$ is small enough, Proposition~\ref{prop:ZeroFree} shows that $T_{\omega}^{n}f$ has no zeros in $U$.  In terms of the local variable $\zeta=\phi(z)$, we have that the function $f(z)=g(\zeta^{1/k})$ can be expanded in a Puiseux series, and then, by Lemma \ref{lem:Divisor},  locally in $U$ we have the equality 
$$\frac{1}{2\pi}\Delta_{\omega} \log \left\vert T_{\omega}^{n}f(z) \right\vert =
- (n+\alpha(n)) \delta_p ,
$$ 
where $\alpha(n)$ is bounded.
Finally, let $\nu$ be the measure given as the sum of all constant parts  $\alpha(n)\delta_p$  together with a similar sum coming from the poles of $f$ at the regular points.
\end{proof}

 In Proposition~\ref{prop:L1LocConvergence} we have shown  that $\dfrac{1}{n}\log \left| \frac{T_{\omega}^{n}f(z)}{n!}\right|$ converges to $\log\frac{1}{\rho(z)}$ in
$L^{1}_{loc} (X^{\ast} \setminus \mathcal{PPL}(\omega,f))$. We deduce that $\Delta_{\omega} \dfrac{1}{n}\log \left| \frac{T_{\omega}^{n}f(z)}{n!}\right|$ converges to
$\Delta_{\omega}\log\frac{1}{\rho(z)}$ as a distribution. Let us prove that the Laplacian of $\log\frac{1}{\rho(z)}$ can be computed in terms of the Cauchy measure of the Voronoi diagram defined by $f$ and $\omega$.

\begin{lemma}\label{lem:LaplaceRho}
We have the equality of distributions on $X^{\ast}$ given by
$$\frac{1}{2\pi} \Delta_{\omega} \log\frac{1}{\rho(z)} = \mu_{\omega,f}-\sum_{p\in \mathcal{PPL}(\omega,f)}\ \delta_p.$$
where $\rho(z)$ is the critical radius of the Voronoi function $g_{z}$ (see Definition~\ref{defn:VorFunc}) and $\mu_{\omega,f}$ is the Cauchy measure of $\mathcal{V}_{\omega,f}$ (see Section~\ref{sub:Cauchy}).
\end{lemma}

\begin{proof} Note first that in  a chart $\zeta=\phi(z)$ the Laplace operator acting on test or real analytic functions such as as  $\log (\frac{1}{\rho(z)})$ can be expressed as 
$$\frac{1}{2\pi }\Delta
_\omega =\frac{2}{\pi }\frac{\partial^2}{\partial \bar\zeta\partial\zeta}.
$$
Thus it suffices to prove Lemma~\ref{lem:LaplaceRho} in a distinguished parameter $\zeta$ and for a corresponding  open set $V$. In a small enough neighbourhood of $p\in \mathcal{PPL}(\omega,f)$, we have that $\rho(\zeta)=\vert \zeta-p\vert$, and $\frac{1}{2\pi }\Delta_\omega \log (\frac{1}{\rho(\zeta)})=-\delta_p$. This gives the Dirac delta function part in the formulation of the lemma. Besides these terms, the support of the Laplacian will be contained in the Voronoi skeleton, since  $\rho$  is harmonic in each 
 (open) Voronoi cell.
Next, take a small neighbourhood $N$  intersecting the Voronoi skeleton and  containing no points of $\mathcal{PPL}(\omega,f)$. Let $p_i\ i=1,\dots$ be the points in  $\mathcal{PPL}(\omega,f)$ such that the relation $\rho(\zeta)=\vert \zeta-p_i\vert$ is valid in a non-zero open subset $O_i$ of $N$. Define $H_i(\zeta):=-\log \vert \zeta-p_i\vert $ which is a harmonic function in the complex plane except for $\zeta=p_i$. Let $\chi_i$ be the characteristic function of $O_i$.  One has the equality of   $L^1_{loc}$-functions given by $M(z):=\log(\frac{1}{\rho(\zeta)})=\sum H_i(\zeta)\chi_i$. 

\medskip
  Let $g$ be a test function with compact support $K$ in a small open set that does not contain the poles of $M$. Then all the $H_i$ are bounded on $K$, and hence in the definition of
  $$
  \langle \Delta M,g\rangle:=\int_Um(z)\Delta g(z)d\lambda,
  $$ we may delete a nullset from $U$. Together with the harmonicity of $M$ this means that we may assume that the support of $g$ contains no vertices (where $v(z)\geq 3)$ of the Voronoi diagram. In other words, no $\delta$ mass appears at the vertices, and all distributional mass off $\mathcal{PPL}(\omega, f)$ is on the Voronoi edges. 
  Then it becomes a question on how to calculate the derivatives of a characteristic function at a smooth segment of the boundary, i. e. recalling the Sokhotski-Plemelj formula.
Let $D_v$ denote the directional derivative associated to a vector $v=(\alpha,\beta)$ in $\mathbb{C} \simeq \mathbb{R}^2$. 

Green's theorem (in its normal form) implies that 
\begin{equation}
\label{eq:normal}
-\langle D_{v} \chi_i,g\rangle=\int\int_{O_i}D_{v}g dxdy=\int_{\partial O_i} (\alpha , \beta)\cdot Ngds,
\end{equation}
where the boundary $\partial O_i$ is oriented so that $O_i$ is located to the left, and $N_i$ is the unit normal pointing away from $O_i$. Note that, in the chosen orientation we get $Nds=(dy,-dx)$. In terms of distributions (and with a slight abuse of notation), this can be formulated as 
\begin{equation}
\label{eq:normal2}
D_{v} \chi_i=-(\alpha, \beta)\cdot N ds=-\alpha dy+\beta dx,
\end{equation}
which remains true even if $v$ is a complex vector.
As a first corollary, Leibniz' rule implies the following equality of distributional derivatives
 \begin{equation}
\label{eq:normal2.5}
\frac{\partial M(\zeta)}{\partial \zeta}=\sum \frac{\partial H_i(\zeta)}{\partial \zeta}\chi_i,
\end{equation}
 since each part of the boundary $V_{ij}$ between the cells $O_i$ and $O_j$ occurs twice with opposite normals in the second part of the sum 
$$
\frac{\partial M(\zeta)}{\partial \zeta}=\sum_i \frac{\partial H_i(\zeta)}{\partial \zeta}\chi_i-\sum_i H_i(\zeta)(\frac{1}{2})(1,-i)\cdot N_i ds,
$$
and $M(\zeta)$ is continuous on $V_{ij}$. (Clearly the endpoints of $V_{ij}$ (at which the normal is undefined) do not contribute to  the integral \eqref{eq:normal}). Thus the distributional derivative coincides almost everywhere with the pointwise derivative (and is also a bona fide $L^1_{loc}$-function).
Applying \eqref{eq:normal2} and once more Leibniz rule  to \eqref{eq:normal2.5} we get
\begin{equation}
\label{eq:normal3}
\frac{\partial^2 M(\zeta)}{\partial \bar \zeta \partial \zeta}=\sum \frac{\partial ^2H_i(\zeta)}{\partial\bar \zeta \partial\zeta}\chi_i- \sum \frac{1}{2}\frac{\partial H_i(\zeta)}{\partial \zeta}(1, i)\cdot N ds,
\end{equation}
since $\frac{\partial}{\partial \zeta}=\frac{1}{2}(\frac{\partial}{\partial x}+i\frac{\partial}{\partial y})$. 
The first sum vanishes by our assumptions on $N$. Concerning the second sum, note that a) $\frac{\partial H_i(\zeta)}{\partial \zeta}=\frac{1}{2(\zeta-p_i)}$, and b) since $Nds=(dy,-dx)$ we get $(1, i)\cdot N \diff s=\diff y-i\diff x=-i\diff \zeta$. 

Each piece of $V_{ij}$ separating the Voronoi cell $O_i$ (located to the left) from $O_j$, will occur twice in the sum, and thus  
\begin{equation}
\label{eq:one}
(-i/4)(\frac{1}{z-p_i}-\frac{1}{z-p_j})\diff z=(i/4)\frac{p_j-p_i}{(z-p_i)(z-p_j)}\diff z.
\end{equation}

Observe that $V_{ij}$ is an interval of the orthogonal midline to  the straight segment between $p_i$ and $p_j$, which is given by the equation
$$
\zeta=(1/2)(p_i+p_j)+ti(p_j-p_i)=(1/2-ti)p_i+(1/2+ti)p_j,\ t\in  \mathbb R.
$$ 
Then, if $\zeta\in V_{ij}$,   
we get $\zeta-p_i=(p_j-p_i)(1/2+ti)$ and $\zeta-p_j=-(p_j-p_i)(1/2-ti)$.  Using the orientation which places $O_i$ to the left of $V_{ij}$, and   $\diff \zeta=i(p_j-p_i)\diff t$, the measure in \eqref{eq:one} is given by 
\begin{equation*}
\frac{\diff t}{4(1/4+t^2)}=\frac{\vert p_i-p_j\vert ^2dt}{4\vert \zeta-p_i\vert \vert \zeta-p_j\vert}. 
\end{equation*}
Here we use the relation $\vert \zeta-p_i\vert \vert \zeta-p_j\vert=\vert \zeta-p_i\vert^2=\vert p_j-p_i\vert^2((1/4)+t^2)$. The standard length measure on $V_{ij} $ is given by $\diff s=\vert p_j-p_i\vert \diff t.$
To finish the proof, it just remains to observe that the latter expression multiplied by $\frac{2}{\pi}$ is exactly the Cauchy measure, see Definitions~\ref{defn:angular} and~\ref{defn:Cauchy}.
\end{proof}

\subsection{Tying up the loose ends}\label{sub:MAINPROOF}

We are now ready to prove our main result. 

\begin{proof}[Proof of Theorem \ref{thm:MAIN}]
In Proposition~\ref{prop:L1LocConvergence}, we have shown that $\dfrac{1}{n}\log \left| \frac{T_{\omega}^{n}f(z)}{n!}\right|$ converges to $\log\frac{1}{\rho(z)}$ in the space of locally integrable functions on $X^{\ast} \setminus \mathcal{PPL}(\omega,f)$. Introducing the Laplace operator $\Delta_{\omega}$ corresponding to the flat metric $|\omega|$ on $X^{\ast}$ (see Section~\ref{sub:Laplacian}), we obtained that  $\Delta_{\omega} \log \left\vert T_{\omega}^{n}f(z) \right\vert^{1/n}$ converges to $\Delta_{\omega} \log\frac{1}{\rho(z)}$ as distributions on $X^{\ast}$.
\par
These distributions have been computed in Lemma~\ref{lem:LaplaceF} and Lemma~\ref{lem:LaplaceRho}. We conclude that $\frac{1}{n} \sum_{z\in Z(T_{\omega}^{n}f)}\ \delta_z$ converges to $\mu_{\omega,f}$ as positive distributions of finite weight on $X^{\ast}$. In other words, the zeros of $T_{\omega}^{n}f$ in the punctured surface $X^{\ast}$ asymptotically distribute according to the Cauchy measure of the Voronoi diagram $\mathcal{V}_{\omega,f}$ (see Section~\ref{sub:Cauchy}). In particular, the Voronoi diagram is contained in the limit set $\mathcal{L}(T_{\omega},f)$.
\par
An immediate computation proves that at a pole $p$  of $\omega$ of order $d$, for any large enough $n$, $T_{\omega}^{n}(f)$ has a zero at $p$ of order $(d-1)n+b$ where $b$ is some constant determined by the order of $f$ at $p$. From this fact we conclude the following:
\begin{itemize}
    \item if $d \geq 2$, then $p \in \mathcal{L}(T_{\omega},f)$;
    \item if $d =1$, then $p \in \mathcal{L}(T_{\omega},f)$ if and only if $p$ is not a pole of $f$.
\end{itemize}
Combining this information with Proposition~\ref{prop:ZeroFree} claiming that any relatively compact open subset in a Voronoi cell is zero-free provided $n$ is large enough, we obtain the complete description of $\mathcal{L}(T_{\omega},f)$ as a subset of $X$, proving claim (i) of Theorem~\ref{thm:MAIN}.

\medskip
It remains to compute the asymptotic root-counting measure of the sequence $(T_{\omega}^{n} (f))_{n \in \mathbb{N}}$. It has been shown  in Corollary~\ref{cor:growthcoefficient} that the sum $\mathcal{Z}_{n}$ of orders  of the poles of meromorphic function $T_{\omega}^{n}(f)$ (and therefore the sum of orders of its zeros) has the asymptotics $\mathcal{Z}_{n} \sim An$ as $n \to +\infty$ where
$$
A = \sum\limits_{z \in \mathcal{PPL}(\omega,f)} (a_{z}+1)
$$
and $a_{z} \geq 0$ is the order of $\omega$ at $z$ (poles of $\omega$ do not belong to the principal polar locus).\newline
Therefore the asymptotic root-counting measure of the sequence $(T_{\omega}^{n} (f))_{n \in \mathbb{N}}$ is  the sum of $\frac{1}{A}\mu_{\omega,f}$ (corresponding to the contribution of the punctured surface $X^{\ast}$) and the term $\frac{1}{A} (d_{p}-1)\delta_{p}$ for each pole of $\omega$ of degree $d_{p} \geq 1$.
\end{proof}

\begin{remark}
The computation of the asymptotic root-counting measure proves in particular that the Cauchy measure $\mu_{\omega,f}$ of the Voronoi diagram $\mathcal{V}_{\omega,f}$ is 
a topological invariant of the pair $(\omega,f)$ in the spirit of the Gauss-Bonnet formula.
\end{remark}

\section{Application to derivatives of algebraic functions}\label{sec:algfunc}

Our approach strengthens and generalises the earlier results of Prather-Shaw in \cite{PrSh1, PrSh2} on the derivatives of (a restricted class of) algebraic functions. 

\begin{notation}
Consider a plane algebraic curve $\Gamma \subset \bC_z\times \bC_w$ with coordinates $(z,w)$ given as the zero locus of  a bivariate polynomial $\Phi(z,w)=0$. Take the projectivization $\widetilde {\Gamma} \subset \bC P^1_z\times \bC P^1_w$ of $\Gamma$ in $ \bC P^1_z\times \bC P^1_w$ with homogeneous coordinates $(Z_0:Z_1), (W_0:W_1)$. Denote the bidegree of $\widetilde {\Gamma}$ by $(k, \ell)$. To avoid trivialities we assume that  both $k$ and $\ell$ are positive. The case of rational functions considered by P\'olya corresponds to the situation $\ell=1,\; k\ge 2$.  
\end{notation}

\begin{notation}In the above notation, consider the projection $\pi_z: \widetilde \Gamma \to \bC P^1_z$. Under the assumption that $\Gamma$ is reduced the preimage $\pi_z^{-1}(\hat z)\subset \widetilde{\Gamma}$ of a generic point $\hat z\in \bC P^1$ consists of $\ell$ distinct points. By the \textcolor{blue}{\textit{branching locus}} $\B(\Gamma)\subset \bC P^1_z$ we mean the set of points $\hat z \in \bC P^1_z$ for which $\pi_z^{-1}(\hat z)\subset \bC P^1_w$ is a positive  divisor of degree $\ell$ with less than $\ell$ distinct points, i.e. some point has multiplicity exceeding $1$. Points  of multiplicity more than $1$ in the preimage $\pi_z^{-1}(\hat z)$ of a branching point $\hat z$ will be called the \textcolor{blue}{\textit{critical points}} of (the meromorphic function) $\pi_z: \widetilde \Gamma \to \bC P^1_z$.

The \textcolor{blue}{\textit{affine branching locus}} $\B_a(\Gamma)$ is not just the restriction of $\B(\Gamma)$ to $\bC_z$, but is its subset for which we additionally require that for $\hat z \in \B_a(\Gamma)$ at least one critical point of the divisor $\pi^{-1}_z(\hat z)$ is affine, i.e. lies in $\bC_w$ and not at $\infty_{w} \in \bC P^1_w$. 

We call a critical point $\hat w\in \pi_z^{-1}(\hat z)\vert_{\bC_w}$ \textcolor{blue}{\textit{essential}} if at least one of the local branches of $\Gamma$ near $\hat w$ has a critical point at $\hat w$ as well, i.e. its projection on a small neighbourhood of $\hat z$ is not a local biholomorphism. (In other words, this branch needs fractional powers  for its presentation  as a Puiseux series at $\hat z$ or, alternatively, $\hat z$ is a branch point of the lift $n\pi_z: N\widetilde \Gamma \to  \bC P^1_z$ where $N\widetilde \Gamma$ is the normalisation of $\widetilde \Gamma$.)  The  branching point $\hat z \in \B_a(\Gamma)$ obtained as projection of an essential critical point $\hat w$ to $\bC P^1_z$ is called \textcolor{blue}{\textit{essential}} as well. Let $\B_a^{ess}(\Gamma)\subset \B_a(\Gamma)$ denote the set of all essential affine branching points. 

Finally, by the {\em locus of poles} $\PP(\Gamma)\subset \bC P^1_z$ we mean the set of points $\hat z \in \bC P^1_z$ for which 
$\pi_z^{-1}(\hat z)\subset \bC P^1_w$ contains $\infty_{w} \in \bC P^1_w$ where $\bC P^1_w= \bC_w\cup \infty$. The {\em affine locus of poles} $\PP_a(\Gamma)\subset \PP(\Gamma)$ is the restriction of $ \PP(\Gamma)$ to $\bC_z$. 
\end{notation}

\begin{definition}
Given $\Gamma$ as above and a positive integer $n$, define the algebraic curve $\Gamma^{(n)}\subset \bC_z\times \bC_w$, called the \textcolor{blue}{\textit{affine $n$-th derivative curve}}, obtained by taking the $n$-th derivatives of all branches of $\Gamma$ considered as (local) algebraic function $w(z)$ with respect to the variable $z$. The \textcolor{blue}{\textit{ $n$-th derivative curve}} of $\Gamma$ is the projectivization $\widetilde{\Gamma}^{(n)} \subset \bC P^1_z \times \bC P^1_w$ of $\Gamma^{(n)}$.
\end{definition}




\begin{definition}\label{defn: zero} We define the {\em (projective) zero locus $\widetilde \ZZ^{(n)}(\Gamma)$ of the $n$-th derivative} of the algebraic function given by the algebraic curve $\Gamma$ as the intersection locus of  $\widetilde {\Gamma}^{(n)}$ with $(\bC P^1_z,0)\subset \bC P^1_z\times \bC P^1_w$. The {\em affine zero locus $\ZZ_a^{(n)}(\Gamma)$} is the restriction of $\widetilde \ZZ^{(n)}(\Gamma)$ to $\bC_z\subset \bC P^1_z$. The {\em (projective) limit set} 
$\F(\Gamma)\subset \bC P^1_z$ is the limiting set of the supports of  the sequence $\{\widetilde \ZZ^{(n)}(\Gamma)\}$ and the  {\em affine limit set} 
$\F_a(\Gamma)\subset \bC_z$ is the limit set of the supports of the sequence $\{\widetilde \ZZ^{(n)}_a(\Gamma)\}\subset \bC_z$.
\end{definition}

\begin{remark}
    In P\'olya's original definition, see \cite{Pol}, the limit set for the usual derivative of a rational function consists of all points  every neighbourhood of which contains points from infinitely many   $\widetilde{\mathcal{Z}}_a^{(n)}(\Gamma)$.
\end{remark}

\begin{problem} Given an affine (resp. projective) reduced algebraic curve $\Ga$ as above, describe its affine (resp. projective) limit set. 
\end{problem}

\begin{remark}
     Obviously,  if the bidegree of $\Ga$ is $(k,1)$ we recover the classical question of P\'olya about the accumulation of zeros of consecutive derivatives of a rational function. Moreover for  algebraic functions given by a certain class of algebraic curves, their limit sets have been described in \cite{PrSh2}. However, to the best of our knowledge, the description of the limit set for general reduced algebraic curves is new. 
\end{remark}

\medskip
To derive results about the limit sets of general algebraic curves, i.e. about the accumulation of the zeros of consecutive derivatives of algebraic functions from the material of the previous sections we proceed as follows. Our main objects are below. 

\begin{definition}
Given a reduced affine curve $\Ga\subset \bC_z\times \bC_w$, take the normalisation $N\widetilde \Gamma$ of the projectivization $\widetilde\Gamma\subset \bC P^1_z\times \bC P^1_w$ and define the meromorphic $1$-form $\tilde{\omega}$ on $\widetilde \Gamma$ as the pullback of the standard meromorphic $1$-form $dz$ on $\bC P^1_z$ under the projection $\pi_z: \widetilde \Gamma \to \bC P^1_z$. Then take the pullback $\omega$ of $\tilde{\omega}$ from $\tilde \Gamma$ to $N\widetilde \Gamma$. The meromorphic function $f: N\widetilde \Gamma \to \bC P^1_w$ is obtained as composition of the normalisation map $n: N \widetilde \Gamma \to \widetilde \Gamma$ with the projection $\pi_w: \widetilde \Gamma\to \bC P^1_w$.  
\end{definition}

\begin{lemma}\label{lm:imp}
In the above notation, the principal polar locus (see Definition~\ref{defn:PPL}) $\mathcal{PPL}(\omega,f)$ of the pair $(\omega,f)$ on Riemann surface $N\widetilde \Gamma$ is given by:
\begin{itemize}
    \item preimage under the normalisation map $n: N \widetilde \Gamma \to \widetilde \Gamma$ of the intersection $\widetilde \Gamma\cap (\bC P^1_z, \infty_w)$ except possibly for the point $(\infty_z,\infty_w$), see Remark~\ref{rmk:infty}; 
    \item preimage under the normalisation map $n: N \widetilde \Gamma \to \widetilde \Gamma$ of the set of all affine essential critical points of $\widetilde \Gamma$.
\end{itemize}
The principal polar locus is non-empty provided $\ell\ge 2$.
\end{lemma}

\begin{remark}\label{rmk:infty}
There is a simple criterion when the projectivization in $\bC P^1_z\times \bC P^1_w$ of an algebraic curve $\Gamma\subset \bC_z\times \bC_w$ given by a bivariate polynomial $\Phi(z,w)=0$ passes through the point $(\infty_z,\infty_w)$. Obviously, this happens if and only if there is a branch of the algebraic function $w(z)$ which tends to $\infty_w$ when $z\to \infty_z$. The latter condition can be reformulated in terms of the Newton polygon $\mathcal N_\Phi$ of $\Phi(z,w)$. Namely, this happens if and only if $\mathcal N_\Phi$ contains an edge with a negative slope and such that $\mathcal N_\Phi$ lies below the line spanned by this edge. This statement can be easily deduced from the known results of Section 38, Th. 63–66 in \cite{Chb} and Ch. 4, Sections 3, Theorem 3.2 and Section 4, Theorem 4.1 in \cite{Wa}. For a more recent exposition, see \cite{WDS}.
\end{remark}

\begin{proof}[Proof of Lemma~\ref{lm:imp}]
Firstly we describe the poles of $f$ and the zeros and poles of $\omega$. Poles of $f$ coincide with $f^{-1}(\infty_w)$ where $\bC P^1_w=\bC_w\cup \infty_w$. In other words, it is the set of points on $N\widetilde \Gamma$ where $\widetilde \Gamma$ intersects the projective line $\bC P^1_z\times \infty_w \subset \bC P^1_z\times\bC P^1_w$.

Next, the singular points on $\widetilde \Gamma$ which project to the affine part $\bC_z\subset \bC P^1_z$ give rise to zeros of $\omega$ on $N\widetilde \Gamma$. We will show a bit later that the singularities of
$\widetilde \Gamma$  projecting to $\infty_z\in \bC P^1_z$ result in poles of $\omega$. 

Indeed each local singular branch of $\widetilde \Gamma$ near a point $z_0\in\bC_z$ can be represented by a Puiseux series 
$$w(z)=\sum_{j \geq \ell} a_j(z-z_0)^{j/k}$$ 
where $k \geq 2$. This branch can be parameterised by $z(t)=z_0+t^{k}$ and $w(t)=\sum_{j=\ell}^\infty a_jt^j$. Here $t$ can be considered as a local chart on $N\widetilde \Gamma$ centered at the point mapped to the singularity of the singular branch of $\widetilde \Gamma$ under consideration. Since $z=z_0+t^k$ one has that $dz=t^{k-1}dt$. Therefore $\omega$ being the pullback of $dz$ acquires a zero of order $k-1>1$ at this point. 

Furthermore let $y=\frac{1}{z}$ be the local coordinate near $\infty_z\in\bC P^1_z$. With respect to $y$, one has $dz=-\frac{1}{y^2}dy.$ For any branch of $\widetilde \Gamma$ which is non-singular in a small neighbourhood of $\infty_z$, i.e. $y=0$ the pullback of $dz$ to this branch will acquire a pole of order $2$.

Finally, let us assume that $\widetilde \Gamma$ has a singular local branch near infinity whose Puiseux series is given by 
$$w(y)=\sum_{j \geq \ell} a_j y^{j/k}.$$
We can parametrise this branch as $y=t^k$ and $w(t)=\sum_{j \geq \ell} a_j t^j.$ The pullback of $dz=-\frac{1}{y^2}dy$ to the local coordinate $t$ on $N\widetilde \Gamma$ equals $-\frac{k}{t^{k+1}}dt,$ i.e. it has a pole of order $k+1>2$ at the respective point.

\medskip

Now by definition, $\mathcal{PPL}(\omega,f)$ consists of (a) poles of $f$ that are not poles of $\omega$ and (b) zeros of $\omega$ at which $f$ is not locally factorised by a primitive of $\omega$. Poles of $f$ which are also poles of $\omega$ are preimages of the point $(\infty_z,\infty_w)\in \bC P^1_z\times \bC P_w^1$. If $\Gamma$ has branches which tend to $\infty_w$ when $z\to \infty_z$ then $\widetilde \Gamma$ contains $(\infty_z,\infty_w)$ and such points exist. Whether $\Gamma$ has such branches is described in Remark~\ref{rmk:infty}. 

\medskip
Finally let us show that at each zero of $\omega$ the function $f$ can not be locally factorised by a primitive of $\omega$. As we have already shown zeros of $\omega$ come from the singularities of $\widetilde \Gamma$ which project to $\bC_z$. Since $\omega$ is the pullback of $dz$ and $f$ is the pullback of the coordinate $w$ the factorisability of $f$ by a primitive of $\omega$ in simple words means that for the original branch of $\Gamma$, the algebraic function $w(z)$ is (locally) holomorphic in $z$ which contradicts the assumption that we consider a singular branch. 

\medskip

In order to prove that the principal polar locus is non-empty, we just observe that for $\ell\ge 2$ the meromorphic function $f: N\widetilde \Gamma \to \bC P^1_w$ is a  ramified cover   branched over at least two points of $\bC P^1_w$. Since at least one of these two points is not $\infty_{w}$ and $\ell\ge 2$, $\omega$ has at least one zero.
\end{proof}

Applying Theorem~\ref{thm:MAIN} we obtain that the zeros of $T^n_\omega(f)$ accumulate on the Voronoi diagram of the Riemann surface $N\widetilde \Gamma$ with the measure described in Section~\ref{sub:Cauchy}. The limit set of the algebraic function given by the curve $\Gamma$ and the respective measure of this limit set on $\bC P^1_z$ are obtained as the push-forward of the respective objects from $\N\widetilde \Gamma$ to $\bC P^1_z$.  

\medskip
More explicitly, we get the following.  

\begin{corollary}\label{th:PolyaAlg} 
For any reduced algebraic curve $\Gamma\subset \bC_z\times \bC_w$, the affine limit set $\F_a(\Gamma)\subset \bC_z$ coincides with the push-forward of the Voronoi diagram of the pair $(\omega, f)$ on $N\widetilde \Gamma$ under the projection $\pi_z$ to $\bC P^1_z$. 
\end{corollary}




\medskip
Let us  present some explicit examples of $\Gamma=\Gamma_0$, $\Gamma_n$ for $n\ge 1$ and illustration of the zeros of consecutive derivatives of the occurring algebraic functions.   

\medskip
\noindent {\bf Example 1.} Let $\Gamma_0 := \Gamma$ be the curve defined by $w^\ell=\frac{P(z)}{Q(z)}$ for some polynomials $P(z)$ and $Q(z)$.
\begin{lemma}\label{lem:recurrence}
	The curve $\Gamma_n := \Gamma^{(n)}$ is given by the equation
	\begin{equation} \label{rational-diff-form} w^\ell = \frac{V_n^\ell(z)}{\ell^{n\ell}P(z)^{n\ell -1} Q(z)^{n\ell + 1}} 
	\end{equation}
	where $V_n(z)$ is a polynomial determined by the recurrence relation
	\begin{itemize}
		\item $ V_0(z) = 1$,
		\item $V_{n+1}(z) = \ell P(z)Q(z)V_n'(z) - ((n\ell-1)P'(z)Q(z) + (n\ell+1)P(z)Q'(z))V_n(z)$.
	\end{itemize}
	In particular, if $P$ and $Q$ has degrees $d_1$ and $d_2$, respectively, then 
	\begin{enumerate}
	\item $V_n(z)$ has degree at most $n(d_1+ d_2 -1)$ and 
	\item $\Gamma_n$ has bidegree at most $(\max\{n\ell(d_1+d_2) - n\ell, n\ell(d_1+d_2)+(d_2-d_1)\}, \ell)$.
	\end{enumerate}
\end{lemma}

\begin{proof}
	The base case $n = 0$ is clear. Assuming $\Gamma_n$ has the form in (\ref{rational-diff-form}) and writing  the denominator in (\ref{rational-diff-form}) as $W_n$, we have
	\[ 	\ell (w^{(n)})^{\ell -1} w^{(n+1)} = \frac{V_n^{\ell -1}(\ell W_n V_n' - V_n W_n')}{W_n^{2}} \]
and
\begin{equation*}
	\begin{aligned}
	(w^{(n+1)})^\ell &= \left( \frac{W_n^{(\ell-1)/\ell}}{\ell V_n^{\ell -1}} \cdot \frac{V_n^{\ell -1}(\ell W_n V_n' - V_n W_n')}{W_n^{2}} \right)^\ell \\
	&= \frac{(\ell W_n V_n' - V_n W_n')^\ell}{\ell ^\ell W_n^{\ell+1}}.
	\end{aligned}
\end{equation*}
Substituting back $W_n$ and simplifying yields
\begin{equation*}
	\begin{aligned}
	(w^{(n+1)})^\ell &= \frac{(\ell^{n\ell}P^{n\ell -2}Q^{n\ell})^{\ell}(\ell PQV_n' - V_n ((n\ell+1)PQ' +(n\ell-1)P'Q)^\ell}{\ell^{\ell} (\ell^{n\ell} P^{n\ell -1}Q^{n\ell +1})^{\ell +1}}	\\
&= \frac{(\ell PQV_n' - V_n ((n\ell+1)PQ' +(n\ell-1)P'Q)^\ell}{\ell^{\ell(n+1)} P^{\ell(n+1) -1}Q^{\ell(n+1) +1}}.
	\end{aligned}
\end{equation*}
The claims about the degrees follow directly. 
\end{proof}

\begin{figure}[!hbt]
\centering
\includegraphics[width=0.45\linewidth]{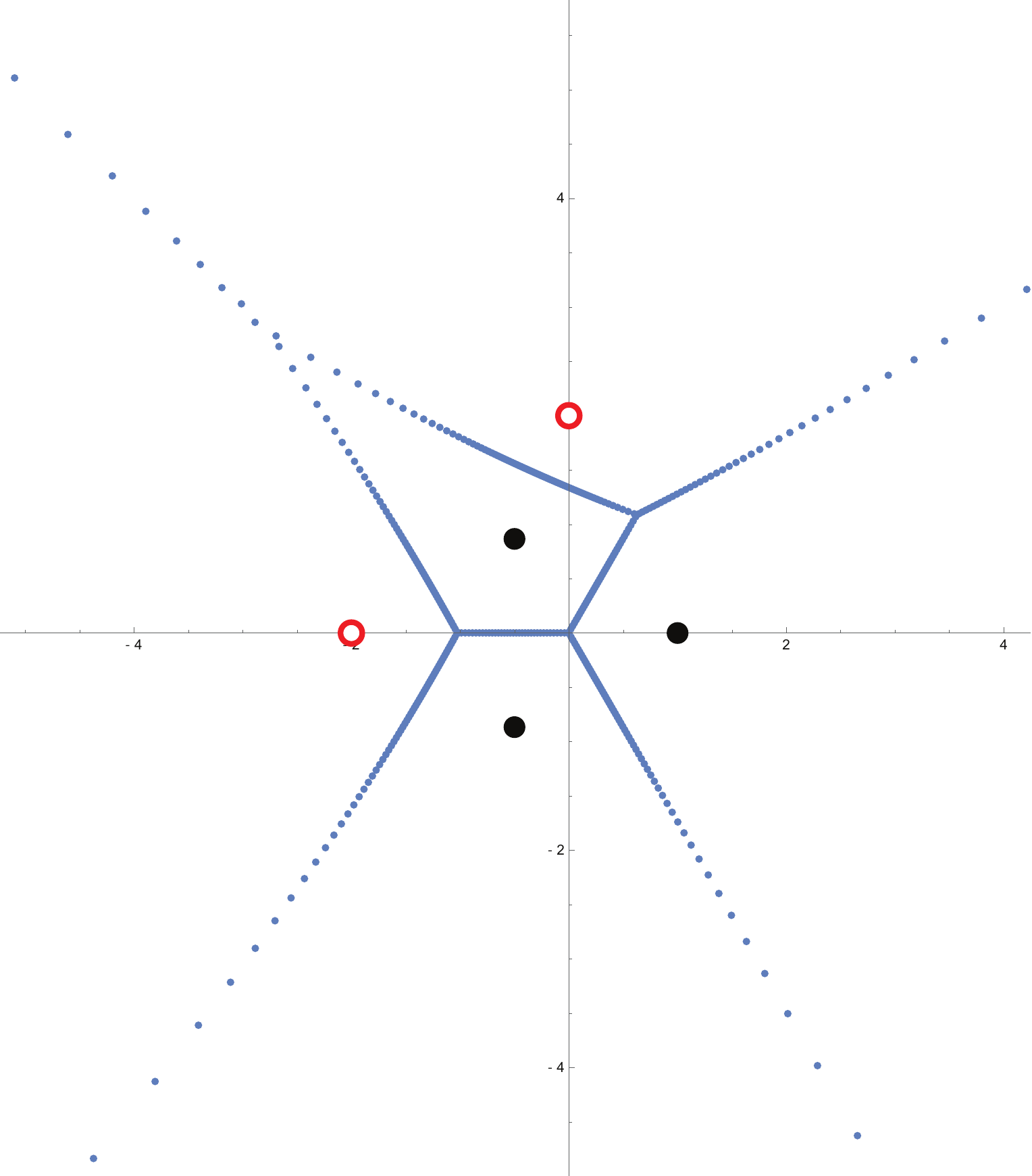}
\caption {Root distribution of $V_{100}(z)$ for $P=(z+2)(z-2i)$, $Q=z^3-1$ and $\ell=3$.}\label{Fig:Ex1}
\end{figure}

\begin{remark} Fig.~\ref{Fig:Ex1} illustrates the above Example 1. In this case the curve $\Gamma$ is given by $(z^3-1)w^3=(z+2)(z-2i)$. It is smooth and its projection onto $\bC P^1_z$ has $6$ branching points at the cubic roots of $1$ (shown by black dots in Fig.~\ref{Fig:Ex1}), $-2$ and  $2i$ (shown by red circles in Fig.~\ref{Fig:Ex1}), and $\infty$. All of them have multiplicity $2$. The cubic roots of $1$ are also poles of the algebraic function. From the Riemann-Hurwitz formula it follows that the genus of $N\widetilde \Gamma$ is $4$. The Voronoi diagram involves only five of the branching points since $\infty$ is a pole of $\omega$.

\end{remark}

\noindent
{\bf Example 2.} As a concrete special case of the above construction take  $\Gamma:=\Gamma_0$ as the unit circle given by $z^2+w^2=1 \leftrightarrow w^2=1-z^2$. 

\medskip
Lemma~\ref{lem:recurrence} implies the following. 
\begin{corollary}
    
\label{lm:circ} The curve $\Gamma^{(n)}:=\Gamma_n$ is given by the equation 
$$w^2=-\frac{U^2_n(z)}{(z^2-1)^{2n-1}}, 
$$
where the polynomial sequence $\{U_n(z)\}$ is given by the  recurrence relation:  
$$U_1(z)=z;\; U_n(z)=(2n-3)zU_{n-1}(z)+(1-z^2)U^\prime_{n-1}(z) \;\text{for \;} n\ge 2.$$
\end{corollary}

\begin{remark}
In the latter case the branching points are located at $\pm 1$ and all roots of $U_n$ are purely imaginary.
\end{remark}

\section{Outlook}\label{sec:Outlook}

\noindent
{\bf 1.} {(Open Riemann surfaces)} The original shire theorem deals with $\mathbb{C}$ which is open and the meromorphic functions on it. In particular, the main result of \cite{Ha} shows that for an entire function of the form $R(z)e^{U(z)}$ with polynomial $U$ a certain part of the total mass of the limiting root-counting measure will be placed at $\infty$. In the present paper we only consider compact Riemann surface $X$, but appropriate modifications of our results work in broader settings.
\par
The crucial hypothesis underlying the constructions of Section~\ref{sec:Geometry} is that our translation surfaces are metrically complete. It follows that the first part of Theorem~\ref{thm:MAIN} still holds for the class of meromorphic functions defined on the surface $X^{\ast}$ punctured at the poles of $\omega$ (and having possibly essential singularities there). The latter class includes the functions studied by Ch.~H\"agg in \cite{Ha}, that globally have a finite number of zeros and whose asymptotic zero counting measure converges.

If the second part of the theorem is reformulated in the following way, then it also will hold for the class of meromorphic functions defined on the surface $X^{\ast}$. Let $U$ be a relatively compact set in $X^\ast$, and let $\mu_n(f, U):=\frac{1}{\vert Z(n,U)\vert }\sum_{z\in Z(n,U)}\delta_z$, where $Z(n,U)$ is the set of zeros of $T_\omega^nf$ in $U$. This is a probability measure, since $Z(n,U)$ is compact and $f$ is meromorphic. The reformulation of the theorem is then that $\mu_n(f, U)$  will converge to the right hand side of 
$$\mu_{\omega,f}+\sum\limits_{p \in \mathcal{P}} (d_{p}-1)\delta_{p},$$
normalised to a probability measure with support in $U$ in the same manner. Then this reformulation covers \cite{Ha} as well and gives a more conceptual proof of similar results for meromorphic functions in the plane in \cite{Ge}.

\medskip
\noindent
{\bf 2.} {(Global geometry of translation surfaces)} An essential feature of the theory of translation surfaces is that the same objects have a complex-analytic side (a Riemann surface  with a holomorphic differential) and a geometric side (a polygon with pairs of sides identified by translations). Although these two descriptions are theoretically equivalent, going from one side to another is a delicate question in practice.
\par
In a translation surface obtained by the gluing of a family of triangles, the lengths and the slopes of the edges are respectively the module and the argument of the periods of some differential over the corresponding relative homology classes. However, starting with a complex structure (defined by a Fuchsian group for example) and an explicit holomorphic differential (in terms of modular forms), it is a difficult problem to determine which relative homology classes are represented by simple geodesic segments (in order to construct a triangulation).
\par
In the current state of the art, the standard approach to obtain a geometric presentation of a translation surface is to discretise the circle of directions and integrate the differential equation corresponding to the differential form to find saddle connections.
\par
If we can call the latter approach ``classical'', Theorem~\ref{thm:MAIN} suggests a ``quantum'' way from the complex-analytic data to the flat picture. For a given translation surface $(X,\omega)$, we consider a meromorphic function $f$ with poles located at the zeros of $\omega$. As $n \to + \infty$, zeros of $T_{\omega}^{n} (f)$ accumulate on the Voronoi diagram defined with respect to the zeros of $\omega$. After an adequate number of iterations, the relative homology classes of the edges of the Delaunay polygonation (dual to the edges of the Voronoi tessellation) are characterised with an arbitrarily low error rate. In Figure~\ref{fig:torus2}, the Voronoi diagram of a nontrivial flat metric in genus one is obtained numerically using this very method.

\medskip
\noindent
{\bf 3.} {(Fuchsian meromorphic connections)} 
Generalisation to an even broader settings can be made as follows. Let $X$ be a compact Riemann surface with a Fuchsian meromorphic connection $\nabla$ on a line bundle $\mathcal{L}$. We can investigate the limit set of a global meromorphic section of $\mathcal{L}$ under iteration of $\nabla$.
\par
Fuchsian meromorphic connections induce complex affine structures (see~\cite{NST}) providing local coordinates where $\nabla$ is conjugated with $\frac{d}{dz}$. We still have a meaningful notion of affine disk immersion so Voronoi diagrams can be defined. Besides, the definition of Cauchy measures in terms of angles is suitable for a generalisation to complex affine structures (see Section~\ref{sub:Cauchy}). Nevertheless, an important difference with the current settings is that in most cases, meromorphic connections can fail to be geodesically complete, as in the case of the Hopf torus $\mathbb{C}^{\ast}/ \langle z \mapsto 2z \rangle$.

\medskip
\noindent
{\bf 4.} {($k$-differentials)} 
A class of Fuchsian meromorphic connection can already be handled with the methods of the current paper. A $k$-differential $\omega$ is a global meromorphic section of the $k^{th}$ tensor power $K_{X}^{\otimes k}$ of the canonical bundle. In local coordinates, it is a complex analytic object of the form $h(z)dz^{k}$ where $h$ is a meromorphic function. A $k^{th}$ root $\omega^{1/k}$ of $\omega$ can be thought as a global meromorphic section of a line bundle twisted by some character $\chi$ valued in the complex multiplicative group $(\mathbb{C}^{\ast},\times)$. The operator $T:f\mapsto \frac{df}{\omega^{1/k}}$ acting on the space of global sections of a suitable line bundle coincides with a Fuchsian meromorphic connection.
\par
For a compact Riemann surface $X$ with a $k$-differential $\omega$, the canonical $k$-cover (see \cite{BCGGM} for details) is the smallest ramified cover $\pi:(\tilde{X},\tilde{\omega}) \rightarrow (X,\omega)$ such that $\tilde{\omega}=\pi^{\ast}\omega$ is the $k^{th}$ power of a globally defined meromorphic $1$-form. This way, the limit set associated with operator $T$ and some meromorphic section $f$ of a line bundle $\mathcal{L}$ is the projection of the limit set associated with an operator $T_{\tilde{\omega}^{1/k}}$ and a section of $\pi^{\ast} \mathcal{L}$ defined on a surface of higher genus $\tilde{X}$. The latter limit set is described in Theorem~\ref{thm:MAIN}.
\par
As in the case of a $1$-form, zeros of $\omega$ appear in the principal polar locus, but one must also include some poles whose order lies strictly between $0$ and $k$. These poles are conical singularities with an angle smaller than $2\pi$ in the singular flat metric and they are ramification values of the canonical $k$-cover.

\appendix

\section{}\label{sec:append}

\subsection{Extending a result of A. G. Orlov}

In \cite[Theorem 1]{Orlov}, A. G. Orlov proved pointwise asymptotics for the coefficients of a Taylor series representation of a branch of an algebraic function, which extends as Puiseux series at singularities. To make the paper self-contained, we include below his result and proof in Theorem~\ref{thm:Orlov}, noting that the result is local, since the proof works mutatis mutandis assuming just that the relevant singularities are algebraic. Then in Corollary~\ref{cor:Orlov} below, we will modify Orlov's proof to obtain uniform convergence for an arbitrary holomorphic function with one Puiseux-type singularity. This is crucial in Section \ref{sec:ZeroFree}, especially in the proof of Lemma~\ref{lem:coeffs-conv}. 

We denote by $D(x,r)$ the open disk of radius $r >0$ centered at $x \in \mathbb{C}$.  

\begin{theorem}\label{thm:Orlov}
Let $f$ be a holomorphic function on an open disk $D(0,\rho)$ with $\nu$ singularities $a_1, \ldots, a_{\nu}$ such that $|a_{s}| = \rho$ for $s = 1,\ldots, \nu$. Suppose that $f$ extends as a Puiseux series to each singular point and holomorphically to every other point of the circle of radius $\rho$. Then, 
\begin{equation}\label{eq: Orlov}
\frac{f^{(k)}(0)}{k!} \sim \sum\limits_{s=1}^{\nu} \tilde{b}_{s}(k)a_{s}^{-k} + O(\rho+\epsilon)^{-k}
\end{equation}
 where $\epsilon>0$ and, for each $s = 1,\ldots, \nu$, the term $\tilde{b}_{s}(k)$ is of the form
$$\tilde{b}_{s}(k) = \frac{B_s}{\Gamma(-m_s)}a_s^{m_s}k^{-(m_s+1)} + o(k^{-(m_s+1)}),$$
where $B_s$ and $m_s$ come from the term $B_s(a_s - z)^{m_s}$ with $m_s$ being the smallest number between the least negative exponent or the least non-integer positive exponent in the Puiseux representation of $f$ at $a_s$.
\end{theorem}

\begin{proof}
Using Cauchy's integral formula, we have:
\begin{equation}\label{eq:cauchy-contour}
g(k):=\frac{f^{(k)}(0)}{k!}=\frac{1}{2i\pi}\int_{C} \frac{f(u)}{u^{k+1}}du,
\end{equation}
where $C$ is the positively oriented circle of radius $0< r<\rho$ centered at $0$.
\par
We will change the contour $C$ of this integration to $\gamma$ defined as in Figure~\ref{fig:contours}. In other words, the contour $\gamma$ is the boundary $\partial D$ of the domain $D=D(0,\rho+\epsilon) \setminus \bigcup\limits_{s=1}^{\nu} D(a_{s},\epsilon)$ for a choice of $\epsilon > 0$ such that the only singularities of $f$ in $D(0,{\rho+\epsilon})$ are $a_1, \ldots, a_\nu$.
\par
\begin{figure}
    \centering
    \includegraphics[width=0.9\linewidth]{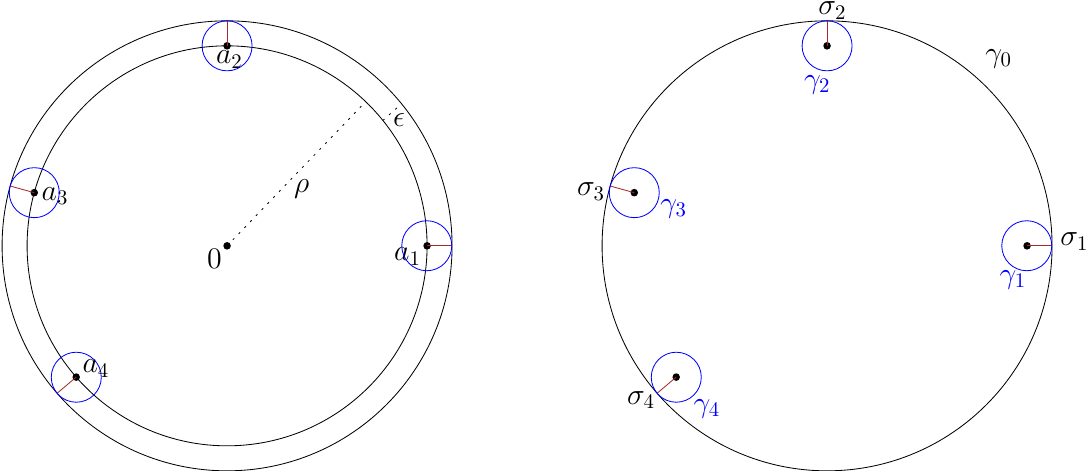}
    \caption{The contour $\alpha$ in Equation \eqref{eq:cauchy-contour} can be replaced by $\gamma_0 \cup \bigcup_{s=1}^\nu \gamma_s$. (The figure displays the case $\nu=4$). On the left, the singularities $a_1, \ldots, a_4$ lie on the circle $|u| = \rho$, and $f$ has no other singularities in $D(0,\rho+\epsilon)$. On the right, the new contour of the integration is the union $\gamma:= \gamma_0 \cup \gamma_1 \cup \cdots \cup \gamma_4$. The cuts $\sigma_s$ are the straight segments joining $a_s$ with $\frac{\rho+\epsilon}{\rho}a_s$.}
    \label{fig:contours}
\end{figure}

Without loss of generality, we assume that the $a_{s}$'s are labeled counterclockwise. Then, the new contour $\gamma$ decomposes into $\gamma_0 \cup \bigcup_{s = 1}^{\nu} \gamma_{s}$ as follows
\begin{itemize}
    \item $\gamma_{s}$ is the boundary of $D({a_{s},\epsilon})$ for $s = 1,\ldots, \nu$;
    \item $\gamma_{0}$ is the union of the $\nu$ circular arcs of radius $\rho + \epsilon$ between the points $\frac{\rho+\epsilon}{\rho}a_{s}$ and $\frac{\rho+\epsilon}{\rho}a_{s+1}$, where the subscripts are considered modulo $\nu$.
\end{itemize}
On $\gamma_{0}$, we have $\left|\frac{f(u)}{u^{k+1}}\right| \leq \frac{K}{(\rho+\epsilon)^{k+1}}$, where $K$ is some constant.
Thus, 
\begin{equation}\label{eq:app-asymp-with-bigO}
g(k)= \frac{1}{2i\pi}\sum\limits_{s=1}^{\nu} \int\limits_{\gamma_{s}} \frac{f(u)}{u^{k+1}}du + O((\rho+\epsilon)^{-k}).
\end{equation}
\par
For $1 \leq s \leq \nu$, we introduce $h_{s}(u)$ which is the Puiseux series of $f$ at $a_{s}$ restricted to the (finitely many) terms with negative exponents. Define  $h(u)=\sum\limits_{s=1}^{\nu} h_{s}(u)$ and take $F(u)=f(u)-h(u)$. 
\par
Denoting by $\sigma_{s}$ the straight segment joining $a_{s}$ to $\frac{\rho+\epsilon}{\rho}a_{s}$, we have that $h$ and $F$ are holomorphic functions on $D(0,\rho+\epsilon) \setminus \bigcup\limits_{s=1}^{\nu} \sigma_{s}$. Thus, to determine the asymptotic of $g(k)$ as in \eqref{eq:app-asymp-with-bigO}, we will study the asymptotic behaviors of the Taylor series coefficients of $h$ and $F$. 

In the rest of the proof, $m$ will denote an arbitrary exponent of one of the considered Puiseux series. In other words, $m$ will be an integer multiple of $\frac{1}{M}$, where $M$ is the least common denominator of the exponents of the Puiseux series at the points $a_{1}, \dots, a_{s}$.
\par
Writing $h_{s}(u)$ as a (finite) sum $\sum\limits_{m<0} c_{s,m}(a_{s}-u)^{m}$, with $m^-_s$ being the least exponent, we can use the $\Gamma$ function to express 
$$b_{s}(k):=\frac{h_{s}^{(k)}(0)}{k!}$$
explicitly as
\begin{align*}
  b_{s}(k) &=\sum\limits_{m<0} \frac{1}{k!} c_{s,m}  (-1)^k(m \cdot (m-1) \cdot \ldots \cdot (m-(k-1))a_s^{m-k} \\
  &= \sum\limits_{m<0} \frac{c_{s,m}}{\Gamma(k+1)} \frac{\Gamma(k-m)}{\Gamma(-m)}a_s^{m-k}.
\end{align*}
Using the property $\lim\limits_{k \to +\infty} k^{m+1}\frac{\Gamma(k-m)}{\Gamma(k+1)} = 1$ of the $\Gamma$ function, we represent $b_s(k)$ asymptotically by the following Puiseux series
\begin{equation}\label{eq: Puiseux-of-frac}
    b_s(k) \sim \sum\limits_{m < 0}   \frac{c_{s,m}}{\Gamma(-m)}k^{-(m+1)}a_s^{m-k},
\end{equation} 
whose term with the least exponent equals
$$\frac{c_{s,m^-_s}}{\Gamma(-m^-_s)}k^{-(m^-_s+1)}a_s^{m^-_s-k}.$$

Next, to compute the contour integral $G_s(k) := \frac{1}{2i\pi}\int\limits_{\gamma_{s}}\frac{F(u)}{u^k}du$, we can replace the contour $\gamma_{s}$ by $\sigma^{+}_{s} \cup \sigma^{-}_{s}$, where both $\sigma^{+}_{s}$ and $\sigma^{-}_{s}$ traverse along $\sigma_s$ but with opposite orientations. Then, in $D(a_s, \epsilon)$, writing 
\begin{equation}\label{eq: taylor-holo-F}
F(u) = \sum\limits_{s=1}^{\nu} \sum\limits_{m \geq 0} d_{s,m} (a_s-u)^{m},
\end{equation}
we have
\begin{align*}
    \frac{1}{2i\pi}\int\limits_{\gamma_{s}} \frac{F(u)}{u^{k+1}}du &= \sum_{m \geq 0} \frac{1-e^{2i\pi m}}{2i\pi}d_{s,m}  \int\limits_{\sigma^+_{s}} \frac{(a_s - u)^{m}}{u^{k+1}}du.
\end{align*} 
We note that the integral in the right-hand side is zero if $m$  is a positive integer. One can rewrite the integrand as
\begin{align*}
    \frac{(a_s - u)^{m}}{u^{k+1}} &= a_s^{m-(k+1)}\left(1- \frac{u}{a_s} \right)^m \exp\left(-(k+1)\ln\left(\frac{u}{a_s}\right) \right)
\end{align*}
and, with the change of variable $t = \frac{u}{a_s} -1$, one has
$$\frac{1}{2i\pi} \int\limits_{\gamma_{s}} \frac{(a_s - u)^{m}}{u^{k+1}}du = \frac{1-e^{2i\pi m}}{2i\pi}e^{-i\pi m} a_s^{m-k}\int\limits_{t=0}^\delta  t^{m} e^{-(k+1) \ln (1+t)}dt$$
where $\delta = \frac{\epsilon}{\rho}$. If we once again change the variable $x = \ln(1+t)$, the problem reduces to the study of the integrals of the form
$$I_{\alpha}(k) := \int_{0}^{\delta_0} x^\alpha \varphi_\alpha(x) e^{-kx}dx,$$
where $\delta_0 = \ln(1+\delta)$ and $\varphi_\alpha(x) = \left(\frac{e^x - 1}{x}\right)^\alpha$. Since $\varphi(x)$ is analytic at $x = 0$, an application of Watson's Lemma yields 
$$I_\alpha(k) \sim \sum_{n \geq 0} \frac{\varphi_{\alpha}^{(n)}(0)}{n!}\frac{\Gamma(\alpha+n+1)}{k^{\alpha+n+1}}  $$
as $k$ tends to $+\infty$ and so $\frac{1}{2i\pi} \int\limits_{\gamma_s} \frac{F(u)}{u^k}du$ is  equivalent to
\begin{align*}
 \sum\limits_{m \geq 0}
  \frac{1-e^{2i\pi m}}{2i\pi}{e^{-i\pi m}}d_{s,m} a_s^{m-k}\sum_{n \geq 0} \left(\frac{\varphi_{m}^{n}(0)}{n!}\frac{\Gamma(m+n+1)}{k^{m+n+1}} \right).
\end{align*}
From the identities $\Gamma(z)\Gamma(1-z)=\frac{\pi}{\sin(\pi z)}$ and 
\begin{align*}
\frac{1-e^{2i\pi m}}{2i\pi}&=\frac{e^{i\pi m}(e^{-i\pi m} - e^{i\pi m})}{2i \pi}=-e^{i\pi m}\frac {\sin \pi m}{\pi },
\end{align*}
we deduce that
\begin{align*}
\frac{1-e^{2i\pi m}}{2i\pi}{e^{-i\pi m}}\Gamma(m+n+1) 
&= \frac{{-}\sin(\pi m)}{\Gamma(-m-n)\sin\left[\pi(m+n+1)\right]} \\
&= \frac{-(-1)^{n+1}}{\Gamma(-m-n)}.
\end{align*}
After simplification, we obtain that $G_s(k)$ is asymptotic to
\begin{align*}
\sum\limits_{m \geq 0}  d_{s,m}  a_s^{m-k}\sum_{n \geq 0} \left(\frac{\varphi_{\alpha}^{n}(0)}{n!}\frac{(-1)^{n}}{\Gamma(-m-n)k^{m+n+1}} \right)
\end{align*}
as $k$ tends to $+\infty$.

Summing over all contours $\gamma_{1},\dots,\gamma_{\nu}$, and reorganizing the latter series, we obtain that $G(k): = G_1(k) + \cdots + G_\nu(k)$ is asymptotic to
 
\begin{equation}\label{eq: holo-asymp}
    \sum\limits_{s=1}^{\nu} 
    \left( 
    \sum\limits_{m,n \geq 0} 
    d_{s,m} 
    a_s^{m}
    \frac{\varphi_\alpha^{n}(0)(-1)^{n}}{n!\Gamma(-m-n)} k^{-(m+n+1)}
    \right) a_{s}^{-k}.  
\end{equation}

Any given value of $m+n+1$ is realised by at most finitely many tuples $(m,n)$ where $n$ is a non-negative integer while $m$ is a non-negative multiple of $\frac{1}{M}$ different from an integer. Therefore, for each $s$, the inner double series can be written as a series indexed by the increasing exponents of $a_s$ whose term with with the least exponent is
$$
\frac{B_{s}}{\Gamma(-m^{+}_s)} k^{-(m^{+}_{s}+1)} a_{s}^{m^{+}_{s}-k},
$$
where $m^{+}_s$ is the least fractional exponent of $(u-a_s)$ in $F(u)$ expressed as in Equation \eqref{eq: taylor-holo-F} and $B_s$ is the coefficient of that term.

Since $f(u)=F(u)+h(u)$, we combine the contributions from the holomorphic part and the Puiseux series part and find the following asymptotic representation of \eqref{eq:app-asymp-with-bigO} 

\begin{equation}\label{eq:g(k)Asymptotic}
g(k) \sim \sum\limits_{s=1}^{\nu} \tilde{b}_{s}(k)a_{s}^{-k} + O((\rho+\epsilon)^{-k}),
\end{equation}
where $\tilde{b}_{s}(k)$ is a Puiseux series in $k$ which is the sum of \eqref{eq: Puiseux-of-frac} and \eqref{eq: taylor-holo-F}. It follows that the asymptotic behavior of the leading term of $\tilde{b}_s(k)$ can be given in the form stated in the theorem which finishes the proof.
\end{proof}

By tweaking Orlov's proof, we will prove the uniform convergence of the first terms of the above asymptotic series that is needed for our purposes. Essentially, this boils down to moving the center of the diagram in Figure \ref{fig:contours}, and hence the whole diagram, inside a small disk around $0$. Note that this changes the branch cut around the singularity that is used.

\begin{corollary}\label{cor:Orlov}
Let $f$ be a holomorphic function on an open disk $D(0,\rho)$ that extends to a Puiseux series at a unique singularity $a$ with $|a|=\rho$ in $D(a, \epsilon')$ for some $\epsilon' > 0$, and holomorphically at every other point of the boundary circle $\partial D(0,\rho)$. Then, there exist $\delta>0$ and $\epsilon>0$ such that the formula
$$
\frac{f^{(k)}(z)}{k!}=b(k,z)(a-z)^{-k}+O((\vert a-z\vert(1+\epsilon))^{-k}).
$$
holds for any $z \in \overline{D}_\delta = \overline{D(0,\delta)}$.
Here, we have
$$
b(k,z)=K (a-z)^{m_0}k^{-(m_0+1)}
(1+
O(k^{-\frac{1}{M}}\vert a-z\vert^{\frac{1}{M}}))
$$
where $K \in\mathbb{C}\setminus\{0\}$ and $M \in \mathbb{N}$ are  constants depending only on $f$ while $m_0=k_0/M$ is the smallest number between the least negative exponent and the exponent of the first non-integer term of the Puiseux series of $f$.
\end{corollary}

\begin{proof}
We follow the notation and strategy of the proof of the previous theorem. First, suppose that $a$ is a pole of $f$. Then, as in \eqref{eq: Puiseux-of-frac}, the singular part $h(z) = \sum\limits_{m_0 \leq m <0} c_m(a-z)^m$ of $f$ has a finite number of terms and has the following asymptotic description of its derivatives
\begin{align*}
    \frac{h^{(k)}(z)}{k!} &\sim \sum\limits_{m_0 \leq m <0}  \frac{c_m}{\Gamma(-m)}k^{-(m+1)}(a-z)^{m-k} \\
    &= (a-z)^{-k}\left[\frac{c_{m_0}}{\Gamma(-m_0)}(a-z)^{m_0}k^{-(m_0+1)}\left(1 + O(k^{-\frac{1}{M}}|a-z|^{\frac{1}{M}}) \right)\right]
\end{align*}
throughout $\overline{D}_{\delta}$, for  any $\delta< \rho$. We subtract $h(z)$ from $f(z)$ 
and can now assume that $f$ satisfies the hypothesis of the corollary, but has a pole-free Puiseux series expansion at $z = a$ given by
\begin{equation}\label{eq: puiseux-at-a}
    f(z) = \sum_{k \geq k_0} c_k(a-z)^\frac{k}{M}.
\end{equation}

First we may assume by shrinking $\epsilon'$ that $f$ may be holomorphically extended to $D(0,\rho+2\epsilon')$, cut at the singularity $a$. This gives us the room to perturb the diagram in figure 11, moving the center of the circle from $0$ to $z$, and the circle $\gamma_0$ to a circle $\gamma(z):=\partial D(z,\vert z-a\vert +\epsilon(z))$. More precisely, let $\delta=\epsilon'/2$, $\alpha=2(\rho+\delta)$  and $\epsilon(z)=\frac{\vert z-a\vert}{\alpha}\epsilon'$. Then 
$$
D(z,\vert z-a\vert +\epsilon(z))\subset D(0,\rho+2\epsilon') \text{ if } z\in \overline{D}_\delta. 
$$
We may now redo the calculations of the theorem. Note that moving the center of $\overline{D}_\delta$ changes the cut $\sigma(z)$ as well as $f$ through the analytic continuation across $\sigma(z)$. In particular, Equation~\eqref{eq:app-asymp-with-bigO} writes as
\begin{equation}
\label{eq:intveryfirst}
\frac{f^{(k)}(z)}{k!}=\frac{1}{2\pi i}
\int_{\gamma_a(z)}\frac{f(u)}{(u-z)^{k+1}}du+
O ((\vert z-a\vert(1+\epsilon))^{-k})
\end{equation}
(with $\epsilon:=\epsilon'/\alpha>0)$. Equation~\ref{eq:intveryfirst} remains true for any $z\in \overline{D}_\delta$, and in fact uniformly, since the implicit constant in the principal term is bounded by the maximum value of (the continuations of) $\vert f\vert $ to any cut disk $D(z, \vert z-a\vert+\epsilon(z)), \ z \in \overline{D}_\delta$. If $0$ is close to the singularity, we make a different choice of $\epsilon(z)$. If $\rho<\epsilon'/8$ and $z\in D(0,\epsilon'/8)$ then $D(z,2\vert z-a\vert)\subset D(a,\epsilon')$ and so we may take $\epsilon(z)=\vert z-a\vert$ and $\epsilon=1$. 

Fixing $z \in \overline{D}_{\delta}$, as in the proof of the theorem, we rewrite the integral 
\begin{equation}\label{eq: decompose-cuts}
I:= \int_{\gamma_a(z)}\frac{f(u)}{(u-z)^{k+1}}du = \int_{\sigma_a^{+}}\frac{f(u)}{(u-z)^{k+1}}du + \int_{\sigma_a^-}\frac{f(u)}{(u-z)^{k+1}}du.  
\end{equation}
Parametrising $\sigma_a^+(z)$ as $u = a+s(a-z)$, $0 \leq s \leq \epsilon$, where $\epsilon = \epsilon(z)/|a-z|$, we obtain
\begin{equation}
\int_{\sigma_a^{+}(z)}\frac{f(u)}{(u-z)^{k+1}}du = \frac{1}{(a-z)^{k}}\int_0^\epsilon \frac{f(a+s(a-z))}{(1+s)^{k+1}}ds
\end{equation}
Changing the variable by setting $1+s = e^{v}$ results in
\begin{equation}\label{eq: change-to-v}
    \frac{1}{(a-z)^{k}}\int_0^{\eta'} f\left(a+(e^v - 1)(a-z)\right)e^{-kv} dv,
\end{equation}
where $\eta' = \log(1+\epsilon) > 0$.

The Puiseux expansion in \eqref{eq: puiseux-at-a}
implies that the function $\tilde{f}(t) := \sum\limits_{k \geq k_0} c_kt^k$ is analytic in a neighbourhood of $t=0$ and $f(z) = \tilde{f}((a-z)^{1/M})$ for a choice of $t = (a-z)^{1/M}$. The function $\tilde{f}$ is just the analytic
continuation of $f$ to the cover $t \mapsto t^{M} = a-z$. In particular, the analytic continuation of $f(z)$ = $\tilde{f}((a-z)^{1/M})$ counterclockwise along a circle with center $z = a$ is given by $\tilde{f}(e^{2\pi i/M}(a-z)^{1/M})$.

Consider $\theta(v):=  \frac{e^v -1}{v}$, which is analytic and non-zero at $v = 0$. Then $\theta^{\frac{1}{M}}(v)$ is defined for $\vert v \vert < 2\pi$ in such a way that it restricts to the $M$-th root which is real and positive for $v \geq 0$ real. Now, define
\begin{align}\label{eq: define-g}
    \tilde{g}_+(w) &:= \tilde{f}\left(w\cdot\theta^{\frac{1}{M}}( w^M / (z-a))\right) \\
     \tilde{g}_-(w) &:= \tilde{f}\left(e^{\frac{2\pi i}{M}}w\cdot\theta^{\frac{1}{M}}( w^M / (z-a))\right).
\end{align}
Thus, along $\sigma_a^+(z)$ we have
\begin{align}
    f(a+(e^v-1)(a-z)) &= \tilde{f}(((v(z-a))^{\frac{1}{M}}\cdot\theta^{\frac{1}{M}}(v)) \notag \\
    &= \tilde{g}_+(((z-a)v)^{\frac{1}{M}}), \label{eq: f-to-g+} 
\end{align}
and, similarly, the analytic continuation of $f$ along $\sigma_a^-(z)$ is 
\begin{equation}\label{eq: f-to-g-} 
f(a+e^{\frac{2\pi  i}{M}}(e^v-1)(a-z)) = \tilde{g}_-(((z-a)v)^{\frac{1}{M}}). 
\end{equation}
By setting $$g(v) = \tilde{g}_+(((z-a)v)^{\frac{1}{M}}) -\tilde{g}_-(((z-a)v)^{\frac{1}{M}}),$$  we can return to the integral $I$ in \eqref{eq: decompose-cuts}. By \eqref{eq: change-to-v}, \eqref{eq: f-to-g+}, and \eqref{eq: f-to-g-}, it equals
\begin{equation}\label{eq: f-to-g}
    I = \frac{1}{(a-z)^k}\int_0^{\eta'} g(v)e^{-kv}dv.
\end{equation}

Application of Lagrange's error term to the first-order Maclaurin polynomial of $g(v)$ proves that the inequality
\begin{equation}\label{eq: lagrange}
   |g(v) - (1-e^{2k_0\pi i/M}) c_{k_0}(a-z)^{m_0}v^{m_0}| \leq D |(a-z)v|^{\frac{k_0+1}{M}} 
\end{equation}
holds for all $0 \leq v \leq \eta'$ and $z \in \overline{D}_{\delta}$, where $c_{k_0}(a-z)^{m_0}, m_0 = \frac{k_0}{M}$, is the first non-integer term of the Puiseux series of $f$, and $D$ is an upper bound for $g^{(k_0+1)}$ on the segment between $0$ and $\eta'$.

By \eqref{eq: f-to-g} and  \eqref{eq: lagrange}, 
\begin{equation}
    |I - (a-z)^{m_0 - k}\int_0^{\eta'}c_{k_0}v^{m_0}e^{-kv} dv| \leq |a-z|^{m_0+ \frac{1}{M} -k} \int_0^{\eta'} D v^{m_0 +\frac{1}{M}}e^{-kv} dv.
\end{equation}
Changing the variable to $u = kv$, we obtain
\begin{align} \label{eq: u-to-kv}
    |I -k^{-(m_0+1)}(a-z)^{m_0-k} \int_0^{k\eta'}c_{k_0}u^{m_0}e^{-u}du| \nonumber \\ 
    \leq k^{-(m_0+\frac{1}{M}+1)}|z-a|^{m_0+\frac{1}{M}-k} \int_0^{k\eta'}  Du^{m_0+\frac{1}{M}}e^{-u}du.
\end{align}
To obtain the desired asymptotic, we apply a standard proof of Watson's Lemma. Extending the interval of integration of the latter integral to the whole positive half-axis, we have a bound
$$ \int_0^{k\eta'} Du^{m_0+\frac{1}{M}}e^{-u}du \leq \int_0^{+\infty} Du^{m_0+\frac{1}{M}}e^{-u}du = D \Gamma(m_0+1/M+1).$$
On the other hand, we have
\begin{align*}
\int_{0}^{k\eta'} c_{k_0}u^{m_0}e^{-u}du &= c_{k_0} \left( \int_0^{+\infty}  u^{m_0}e^{-u}du - \int_{k\eta'}^{+\infty}u^{m_0}e^{-u}du \right) \\
&= c_{k_0}(\Gamma(m_0+1)+O(e^{-k\eta'/2}))
\end{align*}
Here, the above asymptotic term arises as follows. Consider the integral $\int_{k\eta'}^{+\infty}u^{m_0}e^{-u}du$. There exists $r = r(m_0)$ such that $u^{m_0}e^{-u/2} \leq 1$ for all $u \geq r_0$. Hence, if $k\eta' \geq r_0$, i.e. if $k>>0$, 
$$\int_{k\eta'}^{+\infty} u^{m_0}e^{-u} du \leq \int_{k\eta'}^{+\infty} e^{-u/2}du =2e^{-k\eta'/2}.$$
 Corollary~\ref{cor:Orlov} then follows from \eqref{eq: u-to-kv}, noting that the term $O(e^{-k\eta'/2})$ is negligible compared to the other asymptotic terms.
\end{proof}
\begin{remark}
\label{rmk:orlov}
    The constants in the asymptotic estimate in the preceding lemma depend solely on the values of the function and its analytic continuation, and the disk we consider. Note two things in the proof that will be important in the proof of Proposition~\ref{prop:L1LocConvergence}: first that the result is valid in a punctured disk around the singularity. This (together with a compactness argument) means that the result holds (with a suitable $\epsilon$) in a compact subset (punctured at the singularity) of an open set in which all points satisfy the condition of the corollary, in particular for a compact subset of an open Voronoi cell. Secondly, note that in the case with poles, the term with minimal degree $m_0<0$ will give the dominating factor in the expression. Indeed, in that case the result of the lemma may be formulated simpler as
\begin{equation}\label{eq: finaluniform}
\frac{f^{(k)}(z)}{k!}=(a-z)^{m_0-k}k^{-(m_0+1)}(C+B(z,k)),
\end{equation}
where $C\neq 0$ and $B(z,k)$ goes uniformly to $0$ as $k\to + \infty$.
\end{remark}

\subsection{Laplacian of a Puiseux series}
\label{sec:puiseauxlaplacian}

It is known that the Laplacian of the logarithmic absolute value of
a meromorphic function, considered as a $L^1_{loc}$-function (and hence a distribution), is in algebraic-geometric terms the divisor of the function expressed in terms of
Dirac measures (see Theorem 3.7.8 in \cite{Rans}). The same holds for a converging Puiseux series, as we show below.

\begin{lemma}\label{lem:Divisor}
Let $f(z)= \sum\limits_{k \geq r} a_{k}z^{\frac{k}{m}}$ be a converging Puiseux series where $a_{r} \neq 0$. Then, in some cut disk $U:=D(0, R)\setminus [-R,0]$, we have the following equality of $L^1_{loc}$-functions:
$$\frac{2}{\pi }\frac{\partial^2}{\partial \bar z\partial z}\log\vert f\vert=\frac{r}{m}\delta_0.$$
\end{lemma}

\begin{proof}
Writing $f(z)=z^{r/m}g(z)$, we have $\log\vert f\vert=\log\vert g(z)\vert+ \frac{r}{m}\log \vert z\vert$, so it suffices to check that $\frac{\partial^2\log\vert g(z)\vert}{\partial \bar z\partial z}$ vanishes as a distribution.
\par
Let $\psi$ be a test function. Since $g(0) \neq 0$, there exists $M>0$ such $M$ is greater than $\vert \log\vert g(z) \vert \vert \cdot \vert \frac{\partial^2\psi}{\partial \bar z\partial z}\vert $ in the (cut) disk.
\par
For an arbitrary $\eta>0$, we introduce the strip $C(\eta)$ characterised by the inequalities $\Re z\leq \eta$ and $-\eta\leq \Im z\leq \eta$. Then, we subdivide the integral
$$ \int\limits_{U}\frac{\partial^2\log\vert g(z) \vert}{\partial \bar z\partial z} \psi(z)\,dz\wedge d\bar{z} $$
into the integral over $U\setminus C(\eta)$, which vanishes, since $\log\vert g \vert $ is harmonic in this set, and the integral over $U\cap C(\eta)$. Since the absolute value of
$$ \int\limits_{U \cap C(\eta)}\frac{\partial^2\log\vert g(z) \vert}{\partial \bar z\partial z} \psi(z)\,dz\wedge d\bar{z} = - \int\limits_{U \cap C(\eta)}\log\vert g(z) \vert\frac{\partial^2 \psi(z)}{\partial \bar z\partial z}\,dz\wedge d\bar{z} $$ 
is bounded by $2M\eta(R+\eta)$
for any $\eta>0$, this integral also has to vanish. Thus,
$$\int\limits_U \frac{\partial^2\log\vert g(z)\vert}{\partial \bar z\partial z} \psi(z) dz\wedge \bar{dz} = 0 $$
for any test function $\psi$. The claim follows.
\end{proof}

\end{document}